\documentclass[preprint,12pt]{elsarticle}

\makeatletter
\def\ps@pprintTitle{%
 \let\@oddhead\@empty
 \let\@evenhead\@empty
 \def\@oddfoot{}%
 \let\@evenfoot\@oddfoot}
\makeatother

\usepackage{amssymb}
\usepackage{amsmath}
\usepackage{amsthm}
\usepackage[section]{algorithm}
\usepackage{algorithmic}
\usepackage{multirow,multicol}
\usepackage{booktabs}
\usepackage{hyperref}
\usepackage{xcolor}
\usepackage{url,natbib}
\newdefinition{definition}{Definition}
\newdefinition{remark}{Remark}
\newtheorem{theorem}{Theorem}
\newtheorem{lemma}[theorem]{Lemma}
\newtheorem{proposition}[theorem]{Proposition}
\newtheorem{corollary}[theorem]{Corollary}
\newdefinition{example}{Example}
\DeclareMathOperator{\diag}{diag}

\newcommand{\R}{{\mathbb R}}

\newcommand{\C}{{\mathbb C}}
\newcommand{\Cn}{{ \mathbb{C}^n}}
\newcommand{\Cnn}{{ \mathbb{C}^{n\times n}}}

\DeclareMathOperator{\Span}{span}
\DeclareMathOperator{\rank}{rank}
\newcommand{\In}{{ \mathbf{In}}}
\newcommand{\AB}{{(A,B)}}

\begin{document}

\begin{frontmatter}



\title{An indefinite LOBPCG type of algorithm for detecting a definite {H}ermitian matrix pair}


\author{Marija Miloloža Pandur} 
 \ead{mmiloloz@mathos.hr}
 \ead[url]{https://www.mathos.unios.hr/moj_profil/marija-miloloza-pandur/}

\affiliation{organization={School  of Applied Mathematics and Informatics,  Josip Juraj Strossmayer University of Osijek},
            addressline={Trg Ljudevita Gaja 6}, 
            city={ Osijek},
            postcode={31000}, 
            country={Croatia },
            }
\fntext[label3]{ Linear Algebra and its Applications, (2026), v.~746, pp.~111--139, \url{doi 10.1016/j.laa.2026.05.020}}
\fntext[label4]{© 2026. This manuscript version is made available under the CC-BY-NC-ND 4.0 license \url{https://creativecommons.org/licenses/by-nc-nd/4.0/}}

\begin{abstract}
A Hermitian matrix pair $(A,B)$ is called definite if some real linear combination of the matrices $A$ and $B$ is a positive definite matrix. Detection of the definiteness is not straightforward.  We propose a basic subspace algorithm for detecting a large definite matrix pair $(A,B)$ with indefinite $B$. 
 The proposed subspace algorithm is based on iterative testing of small projected Hermitian matrix pairs formed by using  subspaces of small dimensions. Furthermore, we propose a specialized algorithm with parameter $m$, and its preconditioned variant. In the specialized algorithm with $m=3$ we choose the subspaces like in the indefinite locally optimal block preconditioned conjugate gradient (LOBPCG) method.
 Numerical experiments  demonstrate the efficiency of our specialized algorithm, applied on medium-sized pairs, as well as,  on large and banded pairs. Our algorithm very quickly detects (in)definiteness; much faster than some other algorithms.
\end{abstract}


\begin{highlights}
\item New subspace algorithm for detecting definite matrix pair (pencil)
\item Algorithm based on iterative testing of small projected Hermitian matrix pairs
\item Like the indefinite locally optimal block preconditioned conjugate gradient method
\item Quick detection of definiteness of medium-sized or large and sparse pairs
\end{highlights}


\end{frontmatter}



\section{Introduction}\label{sec:intro}
A pair of Hermitian matrices $A,B\in\Cnn$ is called definite  if there exist $\alpha,\beta\in\R$ such that $\alpha A+\beta B$ is positive definite.
 If $(A,B)$ is a definite pair, then the eigenpairs of the generalized eigenvalue problem (GEP) $Ax=\lambda Bx$ can be computed by methods that exploit the definiteness~\cite{Davies2001,MATEJAS2022127358}. 
 Algorithms that solve the partial GEP  of a large definite matrix pair $(A,B)$ with an indefinite $B$ are proposed in~\cite{KressnerMPandurShao2013,MPandur2019}.
 Three other application areas concern setting up a conjugate gradient iteration for  saddle point  linear systems with symmetric indefinite coefficient matrices~\cite{Liesen2008}, checking the hyperbolicity of a Hermitian quadratic matrix polynomial by detecting the definiteness of the associated matrix pair~\cite{htvd2002}, and finally the design of  algorithms for  quadratically constrained quadratic programming~\cite{Adachi2019,Jiang2019,Jiang2018,Taati2019}.

We aim to determine whether a given pair of Hermitian matrices is de\-fi\-nite. There are many algorithms for detecting definiteness, such as the algorithms proposed in~\cite{CHENG199963,Crawford86,CrawfordMoon83,Tisseur2010,htvd2002,Keller1994,Mengi2014,More1993,Veselic1993}, which are suitable for  small- to medium-sized matrix pairs. The algorithms from~\cite{Mengi2018,Mengi2019nons,KressnerLuVan2018,Ding2018} are designed to compute the Crawford number $\gamma$ of a large Hermitian matrix pair (if $\gamma>0$, then the pair is definite). 
 Modifications to  the
algorithms from~\cite{KressnerLuVan2018} that enable faster detection of (in)definiteness  are presented in~\cite{PandurKuzmano}.

In this paper, we propose  new subspace algorithms for detecting definiteness, suitable for medium-sized  or  large and sparse matrix pairs. The proposed subspace algorithms iteratively project a Hermitian matrix pair $(A,B)$ to small dimensional subspaces of $\C^n$ and detect the (in)definiteness of the resulting projected pair. 
 In particular, our algorithms make use of the fact that if the projected pair $(U^HAU, U^H BU)$ is not definite, then the original pair is also not definite. On the other hand, the decision that $(A,B)$  is definite is made by finding $\nu\in\R$ such that  the matrix $A-\nu B$ is either positive definite or  negative definite. In the new Algorithm~\ref{alg:det def} candidates for such $\nu$ are formed from the Ritz values, i.e., from the eigenvalues of the projected pair. Algorithm~\ref{alg:det def} is actually an eigensolver for computing a few inner  eigenvalues and the corresponding eigenvectors, and terminates as soon as the (in)definiteness is confirmed, within a finite number of iterations. The output of Algorithm~\ref{alg:det def}, computed for a definite Hermitian matrix pair, may be used as  input for the algorithms in~\cite{KressnerMPandurShao2013,MPandur2019}. 

 We use the following \textbf{notation}:  $X^H$ is the conjugate transpose of a vector or matrix. $A\succ 0$ ($\succeq 0$) means that $A$ is a Hermitian positive definite (positive semidefinite) matrix; $A\prec 0$ ($\preceq 0$) means that $A$  is a Hermitian negative definite (negative semidefinite) matrix.  $I_n$ ($0_n$) is an identity (zero) matrix of order $n$.  $\diag(X,Y)$ denotes the block diagonal matrix with the first block $X$ and the second block $Y$.   A  diagonal matrix with diagonal elements $d_{1},\ldots,d_{n}$ is denoted by $\diag(d_{1},\ldots,d_{n})$.   The smallest (the largest) real eigenvalue of a Hermitian matrix or a Hermitian matrix pair is denoted by $\lambda_{\min}(\cdot)$ ($\lambda_{\max}(\cdot)$). The integer triplet $\In(B)=(n_+,n_-,n_0)$ is the inertia of a Hermitian matrix $B$, which means that $B$ has $n_+$ positive, $n_-$ negative, and $n_0$ zero eigenvalues, respectively. $\|A\|_2$ is the spectral norm of the matrix~$A$.


 The paper is organized as follows. In section~\ref{sec:prelim} we recall some known results on definite pairs.  A  basic subspace algorithm for detecting the definiteness 
of a  Hermitian matrix pair is derived at the beginning of section~\ref{sec:det def subspaces}. A specialized subspace algorithm, Algorithm~\ref{alg:det def},       is derived and analyzed in section~\ref{subsec: LOBPCG alg} and section~\ref{subsec: analysis}, respectively. Numerical experiments are given in section~\ref{sec:exper}, followed by conclusions in section~\ref{sec:concl}. 

\section{Background}\label{sec:prelim}
In this section, we recall some basic theory about definite pairs and describe algorithms from~\cite{KressnerMPandurShao2013,MPandur2019}.

\subsection{Mathematical preliminaries}

\begin{definition} 
  A Hermitian  matrix pair $(A,B)$ is called \emph{positive} \emph{(negative)} \emph{definite} if there exists a real $\lambda_0$ such that the matrix $A-\lambda_0 B$ is positive \emph{(}negative\emph{)} definite~\textnormal{\cite{Veselic1993}}.
\end{definition}

\begin{remark}\label{rem:definite pair}
Recall that a  Hermitian matrix pair $(A,B)$ is \emph{definite} 
if there exist real constants $\alpha,\beta$ such that $\alpha A+\beta B$ is a positive definite matrix.
In this case, 
\begin{itemize}
    \item[(a)] if $\alpha\neq 0$,  then $(A,B)$ is a positive and/or negative definite pair with  $\lambda_0=-\beta/\alpha$;
    \item[(b)]  if $\alpha= 0$, then $\beta\neq 0$, and $B$ is either a positive or negative definite matrix. The pair $\AB$ with definite $B$ is simultaneously  positive  and negative definite because we can pick $\lambda_0\in\R$ so that the matrix $A-\lambda_0B$ is positive definite and pick some $\nu_0\in\R$ so that the matrix $A-\nu_0B$ is negative definite. 
\end{itemize}
Conversely, if some matrix pair is positive (negative) definite, then it is trivially a definite pair. If the pair $(A,B)$ is positive definite, then the pair $(-A,-B)$ is negative definite, so we present the theory only for positive definite matrix pairs. If the Hermitian matrix pair is not definite, we call it \emph{indefinite}.
 \hfill{$\diamond$}
\end{remark}

A positive definite matrix pair $(A,B)$ is regular  (meaning that $\det (A-\lambda B)$ is not identically zero) and can be diagonalized by a congruence transformation~\cite{StrikoVes1995,Lancaster2005,LiBai2013}. Its   finite eigenvalues are all real and can be ordered as 
\begin{equation}
\label{eq:eigenvalueordering}
\lambda_{n_-}^{-} \leq\dotsb \leq \lambda_{1}^{-} < \lambda_{1}^{+} \leq
\dotsb \leq \lambda_{n_+}^{+},
\end{equation}
where $\In(B) = (n_+,n_-,n_0)$. Each eigenvalue $\lambda_j^+$ is called $B$-positive ($\lambda_j^-$ is called $B$-negative), since it has an eigenvector $x$ that satisfies
$x^H Bx = 1$  ($x^H Bx = -1$). A definite pair $(A,B)$  has $B$-orthogonal eigenvectors for distinct eigenvalues ($x_1^HBx_2=0$ for eigenvectors $x_1,x_2$ belonging to different eigenvalues).
 The matrix $A-\lambda_0 B$ is positive definite for every $\lambda_0\in (\lambda_1^-,\lambda_1^+)$ and nowhere else. The open interval $(\lambda_1^-,\lambda_1^+)$ is called  the \emph{definiteness interval}, and every such $\lambda_0$ is called  a \emph{definitizing shift}.




\begin{theorem}\textnormal{(see \cite[Theorem 2.3]{KressnerMPandurShao2013})}\label{tm:interlacing}
Let $\AB$ be a positive definite matrix pair of order $n$ with finite eigenvalues ordered as in~\eqref{eq:eigenvalueordering} and
$U\in \C^{n\times p}$ such that $\rank\, U=p$.
Then the projected pair  $(U^H AU,U^H BU)$ is also positive definite;
therefore, its finite eigenvalues are real and can be ordered as follows:
\begin{equation}
\label{eq:ordertheta}
\nu_{p_-}^-\leq\dotsb\leq\nu_1^-<\nu_1^+\leq\dotsb\leq\nu_{p_+}^+,
\end{equation}
with $\In(U^H BU) = (p_+,p_-,p_0)$ and  $p_{\pm}\leq n_{\pm}$.
Moreover, the following interlacing inequalities hold:
\begin{align}
\lambda_i^+&\leq \nu_i^+\leq\lambda_{i+n-p+p_0}^+
\quad \text{for} \quad 1\leq i\leq p_+,\label{eq:inter1} \\
\lambda_j^-&\geq \nu_j^-\geq \lambda_{j+n-p+p_0}^-
\quad \text{for} \quad 1\leq j\leq p_-,\label{eq:inter2}
\end{align}
where we set $\lambda_i^+=\infty$ for $i>n_+$ and
$\lambda_j^-=-\infty$ for $j>n_-$.
\end{theorem}

The following  Corollary~\ref{cor: interlacing} is a direct consequence of the Cauchy interlacing inequalities \eqref{eq:inter1} and \eqref{eq:inter2}. 

\begin{corollary}\label{cor: interlacing}
If $\AB$ of order $n$ is positive definite, then for every $U\in \C^{n\times p}$ such that $\rank\, U=p$ and $U^HBU$ is indefinite, the definiteness interval of $(U^H AU,U^H BU)$ contains the definiteness interval of $\AB$.
\end{corollary}

Recall, a Hermitian matrix pair $(A,B)$ is definite if and only if the Crawford number 
\begin{equation} \label{eq: Crawford} 
\begin{split}
  \gamma(A,B)   & :=\min\{|u^H(A+iB)u)|:u\in \C^n,\|u\|_2=1\}\\ 
& = \min\,\{\sqrt{(u^HAu)^2+(u^HBu)^2}: u\in \C^n,\|u\|_2=1\}
\end{split}
\end{equation}
is positive~\cite{STEWART197969}. 
\subsection{Algorithms for Detecting Definiteness of Small- to Medium-Sized Matrix Pairs}\label{sec:existing alg}
This section reviews several algorithms for detecting definiteness, suitable for small- to medium-sized Hermitian matrix pairs of order $n$. They are not suitable for large matrix pairs, since either the computational cost of one iteration  is high or the overall computation is very time-consuming.

In 1993, Mor{\'e}~\cite{More1993} proposed a bisection algorithm on the real line for computing a real $\lambda$ such that $A-\lambda B$ is  positive definite, or to determine that such 
$\lambda$ does not exist. At each iteration, the algorithm performs an attempted Cholesky factorization to test
the definiteness of some linear combination of the given matrices. 
The algorithm provides no criterion for declaring a given pair indefinite, other than determining whether the length of the current interval is smaller than a prescribed tolerance $\varepsilon$. The number of iterations  $j$ needed to achieve the  required tolerance $\varepsilon$, is bounded by
$$\displaystyle j\leq \Big\lceil \log_2\frac{\varepsilon_0}{\varepsilon}\Big\rceil,$$
with an initial bracket size $\varepsilon_0=\lambda_u-\lambda_l$, where $(\lambda_l,\lambda_u)$ is an initial finite interval that contains the definiteness interval if the pair is definite,  and the required bracket size $\varepsilon\leq \varepsilon_0$. So, for example, in an indefinite case with $\varepsilon_0=10$ and $\varepsilon=10^{-15}$ the algorithm from~\cite{More1993} performs 53 attempted Cholesky factorizations. So, it can be very time-consuming.

\medskip 
In 1993, Veseli\'c~\cite{Veselic1993} proposed a Jacobi-type eigenreduction algorithm (the so called $J$-Jacobi algorithm) for  definite matrix pairs $(A,J)$ of real symmetric matrices where $J=\diag(I_{n^+},-I_{n^-})$ is a diagonal matrix of signs. The algorithm uses $J$-orthogonal congruences to diagonalize the pair  $(A,J)$. 
If  it is app\-lied to  a symmetric  pair whose definiteness is unknown, the $J$-Jacobi either converges or discovers the indefiniteness in finitely many steps~\cite[Theorem 2.4]{Veselic1993}.
 The $J$-Jacobi algorithm is asymptotic quadratic convergent~\cite{HariDrmac,Matejas2006434} and relatively accurate~\cite{Slapnicar2003387,SlapnicarPhD} when applied to $(A,J)$ with positive definite $A$.  The  algorithm can also be applied to a complex Hermitian matrix pair $(A,J)$ (e.g.,~\cite[Section~2]{SINGER20125704}).

 We cannot directly apply the $J$-Jacobi algorithm to a general Hermitian matrix pair $(A,B)$ with nonsingular $B\neq \diag(\pm 1)$. A preprocessing step is required (see the last paragraph of section~\ref{subsec: LOBPCG alg}) to construct an auxiliary matrix pair $(H,J)$, which for a dense matrix costs $11n^3/3$ flops. After that, we apply the $J$-Jacobi method to this pair, with a computational cost of $4 s n^3$ flops, or potentially less~\cite[Section 4]{Veselic1993}, where $s$ denotes the number of sweeps performed. 


\medskip
In 1994, Keller~\cite[Section 3.6.2]{Keller1994} introduced a coordinate-relaxation method for detecting definiteness of a matrix pair, based on~\cite[Corollary 3.7]{StrikoVes1995}, which uses the subspaces of dimension 2. The algorithm is very simple and can be employed for a matrix pair even when matrix factorizations are not feasible. Although each iteration is inexpensive (requiring only $O(n^2)$ flops), verifying that a given $\AB$ is definite demands the computation of both $\lambda_1^-$ and $\lambda_1^+$, which becomes very time-consuming for large definite pairs $\AB$. The method uses two  criteria to decide whether a given pair is indefinite.

\medskip
In 2002, Higham, Tisseur and Van Dooren~\cite[Algorithm 2.3]{htvd2002} proposed a bisection algorithm for computing the interval that contains the Crawford number of a given Hermitian matrix pair.  
Each iteration of this algorithm is computationally expensive, since it involves computing all eigenvalues of a non-Hermitian quadratic eigenvalue problem of order $n$, together with  up to $2n$ computations of the smallest eigenvalue of a Hermitian matrix of order $n$. In the same paper~\cite[Algorithm 2.4]{htvd2002}, a more efficient approach is proposed: the level set algorithm.  This algorithm computes all $2n$ eigenvalues of just one non-Hermitian quadratic eigenvalue problem of order $n$, and then verifies one simple condition to decide wheather a given pair is definite.

\medskip
In 2010, Guo, Higham and Tisseur proposed the improved arc algorithm~\cite[Algorithm~2.3]{Tisseur2010}.  
This method is a bisection-like algorithm on the unit circle for finding a real value $t$ for which $B(t):=A\sin t + B\cos t$ is positive definite, or to  conclude that no such $t$ exists. At each iteration, the improved arc algorithm performs an attempted
Cholesky factorization  to check the definiteness of $B(t)$ at the current  value of $t$. It has two criteria to declare that a given pair is indefinite: if it finds some unit vector $z$ such that $z^HAz=z^HBz=0$, or the length of the arc is greater than or equal to $\pi$. The algorithm is backward stable and linearly convergent. 
The efficiency of the  algorithm significantly depends on the way how the Cholesky factorization is performed: the authors of~\cite{Tisseur2010} recommend the Cholesky factorization with complete pivoting. The arc algorithm can, in principle, be applied to large and sparse matrix pairs (when combined with a sparse Cholesky factorization routine), but it is frequently very time-consuming in practice.



\subsection{Algorithms for Detecting Definiteness of Large-Scale Matrix Pairs}\label{sec:existing alg large}

The following eigenvalue optimization formula
\begin{equation}\label{eq:Crawford global max}
\displaystyle \gamma(A,B)=\max\Big\{\max_{\theta\in[0,2\pi]} \lambda_{\min} (A\cos\theta+B\sin\theta),0  \Big\},
\end{equation} 
holds~\cite[Theorem 2.1]{CHENG199963}, \cite[Equation (2.8)]{htvd2002}.
In 2018, Kressner, Lu, and Vandereycken proposed in~\cite{KressnerLuVan2018} a subspace method,\footnote{This method is derived from a more general algorithm described in~\cite[Algorithm~1]{Mengi2018}; see also~\cite[Algorithm~2]{Mengi2019nons}.} which is based on~\eqref{eq:Crawford global max}, for computing the Crawford number of a general large matrix $A+iB$ with Hermitian matrices $A$ and $B$.
In each iteration,~\cite[Algorithm~1]{KressnerLuVan2018} solves a Hermitian linear eigenvalue problem of order $n$ for the smallest eigenvalue, and a corresponding eigenvector. That eigenvalue is a lower bound for the Crawford number, while the Crawford number of a projected pair is an upper bound. The algorithm has a local convergence of order $1+\sqrt{2}$, and is applicable when the matrices are too large to fit completely in memory or are given impli\-ci\-tly through procedures $x\mapsto Ax$ and $x\mapsto Bx$ for a given vector $x$. In the context of definiteness detection, the author and Kuzmanovi\'c Ivi\v{c}i\'c proposed Algorithm~4.2 in ~\cite{PandurKuzmano} as the modification of~\cite[Algorithm~1]{KressnerLuVan2018} in a way that the algorithm stops when the smallest eigenvalue $\lambda_{\min} (A\cos\theta+B\sin\theta)$ for a particular $\theta$, is positive (indicating definiteness of $(A,B)$), or when one of several conditions for indefiniteness is fulfilled. Similarly,~\cite[Algorithm~4.3]{PandurKuzmano} is a modification of the 3-vector subspace method~\cite[Algorithm~2]{KressnerLuVan2018}. Although the algorithms in~\cite{PandurKuzmano} may terminate in the first iteration, they require more execution time than Algorithm~\ref{alg:det def} in the experiments presented in section~\ref{sec:exper}.

\medskip

In 2020, Lu proposed~\cite[Algorithm~5.1]{Ding2018} that solves the eigenvalue optimization problem~\eqref{eq:Crawford global max} by reformulating it into a specific  nonlinear eigenvalue problem with eigenvector nonlinearity (NEPv)
$H(x)x=\lambda x$, $x\neq 0$,
where $H(x)=g_1(x)A+g_2(x)B$, and $g_1,g_2$ are specific functions.
Algorithm~5.1 in~\cite{Ding2018} is a nonlinear  locally optimal block preconditioned conjugate gradient (LOBPCG) type of algorithm that aims to compute the smallest eigenvalue $\lambda_{\min}(H(x))$  and a corresponding eigenvector $x$. The algorithm stops when the norm of the residual $(H(x_k)-\mu_kI)x_k$ becomes smaller than a specified tolerance, where $\mu_k$ denotes the $k$th approximation of the desired eigenvalue, which in turn approximates the Crawford number of the pair $(A,B)$. In some cases, the algorithm terminates because the eigenvalue approximations lose monotonicity. However, the norm of the current residual may still exceed the specified tolerance, and the Crawford number is therefore not computed to the desired precision. In the context of definiteness detection, we modify~\cite[Algorithm~5.1]{Ding2018} and terminate it when the approximation of the Crawford number falls below a specified tolerance, indicating that the original matrix pair is close to being indefinite. This algorithm is also applicable when the matrices are too large to fit completely in memory or are given implicitly through procedures $x\mapsto Ax$ and $x\mapsto Bx$ for a given vector $x$. 

\subsection{Indefinite Preconditioned Gradient Iterations}\label{subsec: grad iter}

Preconditioned eigenvalue solvers for definite matrix pairs $(A,B)$ with indefinite $B$ are proposed in~\cite{KressnerMPandurShao2013,MPandur2019}.  These algorithms compute a small number of eigenvalues that are closest to the definiteness interval $\mathcal{I}$, and the corresponding eigenvectors. More precisely, they compute $k_+$ smallest $B$-positive eigenvalues that are on the right side of $\mathcal{I}$,  and $k_-$ largest $B$-negative eigenvalues that are on the left side of $\mathcal{I}$ of a given positive definite matrix pair. Since the desired eigenvalues are inner eigenvalues of some definite matrix pair, a shift  $\lambda_0$ from the definiteness interval $\mathcal{I}$ must be known to construct an effe\-ctive preconditioner $T$.  
The convergence of the algorithms in~\cite{KressnerMPandurShao2013,MPandur2019} heavily relies on a good choice of a  preconditioner. The analysis in~\cite{KressnerMPandurShao2013,MPandur2019} for a positive definite matrix pair suggests that choosing $T_+\approx (A-\lambda_0^+B)^{-1}$ with a definitizing shift $\lambda_0^+$ close to $\lambda_1^+$ yields reasonable convergence for the smallest $B$-positive eigenvalues. Similarly, a preconditioner $T_-\approx (A-\lambda_0^-B)^{-1}$ with a definitizing shift $\lambda_0^-$ close to $\lambda_1^-$ yields reasonable convergence for the largest $B$-negative eigenvalues. So it would be difficult to find a shift that works equally well for $B$-positive and $B$-negative eigenvalues simultaneously. Therefore, algorithms that use two shifts and two preconditioners are also proposed  in~\cite{KressnerMPandurShao2013,MPandur2019}.  The numerical experiments  show that the algorithms with two preconditioners are much more efficient than the algorithms with one preconditioner. The shifts from the definiteness interval $\mathcal{I}$ yield  positive definite preconditioners, for which there exists a convergence theory. Nevertheless, the shifts outside of $\mathcal{I}$, but close to its endpoints, which produce indefinite preconditioners, could also be very efficient, as numerical experiments demonstrate.

The algorithms in~\cite{KressnerMPandurShao2013,MPandur2019} produce a sequence of matrices
$$X_0,X_1,X_2,\ldots\in\C^{n\times k},\qquad k=k_++k_-,$$ as follows. In the $i$-th iteration of the $(m)$-scheme method with one preconditioner, $m\geq 2$, from~\cite{MPandur2019} the subspace
\begin{equation}\label{eq: subspace basis}
    \Span U_i=\Span [X_i,W_i,X_{i-1},X_{i-2},\ldots,X_{i-m+2}]
\end{equation}
is considered (a matrix $X_{-j}$ is an empty matrix), where
$W_i=TR_i:=T(A X_i-BX_iD_i)$
is the preconditioned residual for some Hermitian  matrix  $T\in\Cnn$, and 
$D_i:=(X_i^HBX_i)^{-1}X_i^HAX_i$.
Then, the standard Rayleigh-Ritz procedure is performed
 by computing the eigenpairs of the projected pair $(U_i^HAU_i,U_i^HBU_i)$ with order eigenvalues 
$$\nu_{i+1;p_-}^-\leq \cdots\leq\nu_{i+1;1}^-<\nu_{i+1;1}^+\leq \cdots\leq \nu_{i+1;p_+}^+,$$
where $\In(U_i^HBU_i)=(p_+,p_-,p_0)$.
Then 
\begin{equation}
\begin{array}{cl}\label{eq: D i Y}
  D_{i+1}  = & \diag(\nu_{i+1;k_-}^-,\ldots,\nu_{i+1;1}^-,\nu_{i+1;1}^+, \ldots,\nu_{i+1;k_+}^+), \\
   Y_{i+1}  =& [y_{i+1;k_-}^-,\ldots,y_{i+1;1}^-,y_{i+1;1}^+,\ldots,y_{i+1;k_+}^+]
\end{array}
\end{equation}
is returned, where $y_{i+1;j}^\pm$ is  $U_{i}^H B U_{i}$-normalized\footnote{A vector $z$ is $M$-\emph{normalized} if $|z^HMz|=1$ for a given Hermitian matrix $M$.} eigenvector of the projected pair $(U_i^HAU_i,U_i^HBU_i)$ belonging to $\nu_{i+1;j}^\pm$. The Ritz values and vectors are  given as $\nu_{i+1;j}^\pm$ and $U_{i}y_{i+1;j}^\pm$, respectively. So, the  matrix $X_{i+1}=U_iY_{i+1}$ of the Ritz vectors has $B$-orthonormal columns.\footnote{We call vectors $x_i$ and $x_j$  $B$-\emph{orthonormal} if $|x_i^HBx_j|=\delta_{ij},$ where $\delta_{ij}$ is the Kronecker delta symbol.} 

 For $m=3$ we have the indefinite locally optimal block preconditioned conjugate gradient (LOBPCG) method from~\cite{KressnerMPandurShao2013}, and for $m=2$ we have the so-called indefinite block preconditioned steepest descent/ascent (BPSD/A) method (a truncated version of LOBPCG). 
In the version with two preconditioners $T_+$ and $T_-$, a  preconditioned residual is formed by splitting the residual $R$ into two parts: $R_{+}$ and $R_{-}$ associate with $B$-positive Ritz values and $B$-negative Ritz values, respectively, so $W=[T_+R_{+},T_-R_{-}]$.

\section{Subspace algorithms}\label{sec:det def subspaces}

 In this section, we introduce algorithms for detecting the (in)defini\-te\-ness of a given Hermitian matrix pair $(A,B)$.  
We begin by presenting the basic subspace algorithm, while a specialized version of the algorithm is discussed in Section~\ref{subsec: LOBPCG alg}.
 The proposed algorithms use several criteria to declare indefiniteness, whereas definiteness is confirmed through a successfully completed Cholesky factorization. In the definite case, our algorithms yield a single definitizing shift along with a crude estimate of the definiteness interval. 

\begin{lemma}\label{re: B ind}
If the matrix $B$ is indefinite, then the pair $\AB$ cannot be simultaneously a positive definite and a negative definite pair.
\end{lemma}
\begin{proof}
     Let $\AB$ be a positive definite pair. When $B$ is indefinite, we have at least one $B$-positive eigenvalue and at least one $B$-negative eigenvalue of $(A,B)$. Then for any $\lambda\in (-\infty,\lambda_1^-)\cup(\lambda_1^+,+\infty)$ the matrix $A-\lambda B$ is indefinite by~\cite[Proposition 4.1]{StrikoVes1995} and~\cite[Lemma 3.8]{LiBai2013}. 
\end{proof}

\begin{theorem}\textnormal{(see \cite[Theorem 3.10]{StrikoVes1995} and~\cite[Theorem 2.2]{LiBai2013})}\label{tm: intersection}
Let $\AB$ be  a Hermitian matrix pair of order $n$, and suppose that it is regular. Suppose $1\leq k_+\leq n_+$ and $1\leq k_-\leq n_-$ and that the  definiteness intervals of the projected pairs $(U^H AU,J_k)$, taken for all $U$ satisfying $U^H BU=J_k:=\diag(I_{k_+},I_{k_-})$, have a nonvoid intersection $\mathcal{I}$. Then $(A,B)$ is positive definite, and $\mathcal{I}$ is the definiteness interval of $\AB$. 
\end{theorem}


\begin{lemma}\label{lemma:pos-neg}
   Let $\AB$ be a regular Hermitian matrix pair of order $n$ with an indefinite $B$, that is, $n_-\geq 1$ and $n_+\geq 1$. 
The pair $\AB$ is positive (negative) definite if and only if the projected pairs $(U^H AU,U^H BU)$ are  positive (negative) definite for all matrices  $U\in \C^{n\times p}$ with full column rank,  such that $\In(U^H B U)\geq (1,1,0)$, where the inequality is understood element-wise, and for all $p=1, \ldots,n$.  
\end{lemma}
\begin{proof}
Let $A-\lambda_0 B\succ 0$.  If $U$ is of full column rank, then for every $y\neq 0$ holds $0\neq Uy=:x$. Therefore, 
$$y^H U^H (A-\lambda_0 B)  U y=x^H (A-\lambda_0 B)x>0$$
meaning $U^H (A-\lambda_0 B) U$ is a positive definite matrix. The additional assumption  $\In(U^H B U)\geq (1,1,0)$ means that $U^H B U$ is indefinite, and by Lemma~\ref{re: B ind} the pair $(U^H AU,U^H B U)$ cannot simultaneously be positive and negative definite. So, this projected pair is only positive definite. The proof is analogous for a negative definite case. For the sufficient condition, just take $U=I_n$. 
\end{proof}

  If $A$ or $B$, or both, are definite matrices, then the pair $\AB$ is trivially definite. So, from now on, we assume that $B$ is indefinite. Corollary~\ref{cor: interlacing}, Theorem~\ref{tm: intersection} and Lemma~\ref{lemma:pos-neg} give a hint on how to use projected pairs $(U^HAU,U^HBU)$ in  deciding whether a given Hermitian pair $(A,B)$ is definite or not. We now describe the \textit{\textbf{basic subspace algorithm}} for detecting definiteness (cf.~\cite[Algorithm~4.1]{PandurKuzmano}). We  iteratively choose a subspace $\Span U$ such that $U^HBU$ is indefinite and forms a projected pair $(U^HAU,U^HBU)$. If some projected pair is indefinite, so is the original pair $(A,B)$  by Theorem~\ref{tm:interlacing}. If the de\-fi\-ni\-teness intervals of two definite projected pairs have a void intersection, then the original pair is indefinite according to Theorem~\ref{tm: intersection}. If one projected pair is positive definite and another is negative definite, then the necessary condition of Lemma~\ref{lemma:pos-neg} is violated, and the original pair is indefinite. The definiteness interval of any definite projected pair contains the definiteness interval of the original definite pair by  Corollary~\ref{cor: interlacing}. If $A-\nu B$ is a positive or negative definite matrix for some number $\nu$ from the intersection of the definiteness intervals of all previously formed definite projected pairs, then the original pair $(A,B)$ is definite by Remark~\ref{rem:definite pair}. We  also use 
 \textit{a sufficient condition} for indefiniteness: if $z^H Az=z^H B z=0$ for some  $z\in\Cn$ such that $\|z\|_2=1$, then $\AB$ is indefinite. Since in this case  there are no real numbers $\alpha,\beta$ such that $z^H(\alpha A + \beta B)z>0$ for all $z\in\Cn$, $z\neq 0$, that is, there is no real  linear combination of the matrices $A$ and $B$ that is a positive definite matrix.

\subsection{A BPSD/A and LOBPCG type of algorithms}\label{subsec: LOBPCG alg}

The main question in the basic subspace algorithm is how to choose new search subspaces.
 We propose to use subspaces as in~\cite{KressnerMPandurShao2013,MPandur2019}. As mentioned in section~\ref{subsec: grad iter}, these algorithms use definite projected pairs and give their definiteness intervals. For simplicity, we write the algorithm only for $m=2$ (BPSD/A type) and $m=3$ (LOBPCG type). So, the specialized algorithm for detecting definite pairs is given as Algorithm~\ref{alg:det def}. It is an interior eigensolver for computing $k_+$ smallest $B$-positive eigenvalues and $k_-$ largest $B$-negative eigenvalues, and the corresponding eigenvectors, if the given matrix pair $(A,B)$ is positive definite. It terminates as soon as the definiteness is confirmed, although the eigenpairs have not been computed to a reasonable accuracy. If the given $(A,B)$ is indefinite, Algorithm~\ref{alg:det def} determines indefiniteness within a finite number of iterations.

\begin{algorithm}
\footnotesize
\caption{A  subspace algorithm for detecting a definite matrix pair}
\label{alg:det def}
\begin{algorithmic}[1]
\REQUIRE Hermitian matrices $A$, $B\in\C^{n\times n}$ such that $B$ is indefinite; integers $k$, $k_\pm\geq 1$; an initial matrix
$U_0\in\C^{n\times k}$ s.t. $\rank\,U_0=k\geq k_++k_-$ and $(k_+,k_-,k_0)\leq \In(U_0^HBU_0)$, where $k_0=k-(k_++k_-)$; an integer $m=2$ or $m=3$; tolerances $\mathsf{tol}>0$ and $\mathsf{tol_{ind}}>0$.
\ENSURE (if definiteness  is detected) a matrix $X_{i+1}$, an interval $(\nu_{i+1;1}^-,\nu_{i+1;1}^+)$ and $\lambda_0=\nu_{i+1}$ such that $A-\lambda_0B$ is  definite.
\STATE\label{line1} $DEF(-1)=0$;  

\FOR{$i=0,1,2,\ldots$}
  \IF{the sufficient condition~\ref{line:suf1} or~\ref{line:suf2} for some column of $U_i$ is satisfied} \label{line2}
  \STATE STOP: $(A,B)$ is indefinite or close to an indefinite pair;\label{line3}
\ENDIF
 \STATE orthonormalize $U_i$; (or $B$-orthonormalize $U_i$; )\label{line6}
  \STATE check the definiteness of $(U_i^HAU_i,U_i^HBU_i)$;\label{line7}
  \IF{$(U_i^HAU_i,U_i^HBU_i)$ is indefinite}\label{line8}
    \STATE STOP: $\AB$ is indefinite.\label{line9}
  \ENDIF 
    \IF{$(U_i^HAU_i,U_i^HBU_i)$ is positive definite}\label{line11}
         \STATE $DEF(i)=1$;\label{line12}
    \ELSE
         \STATE $DEF(i) = -1$;\label{line14}
    \ENDIF
    \IF{$DEF(i)\cdot DEF(i-1)<0$}\label{line16}
         \STATE STOP: $\AB$ is indefinite.\label{line17}
    \ENDIF 
    \STATE Apply the Rayleigh-Ritz procedure on $(U_i^HAU_i,U_i^HBU_i)$ to get inner $k_-+k_+$ eigenpairs. Form $(D_{i+1},Y_{i+1})$ from~\eqref{eq: D i Y};
       \label{line19}
        \IF{$|\nu_{i+1;1}^+-\nu_{i+1;1}^-|<\mathsf{tol}$}\label{line23}
            \STATE STOP: $\AB$ is indefinite or close to an indefinite pair; \label{line24} 
        \ENDIF
            \STATE let $\nu_{i+1}=(\nu_{i+1;1}^-+\nu_{i+1;1}^+)/2$;\label{line26}
            \IF{$DEF(i)=1$ and $A-\nu_{i+1}B\succ 0$ }\label{line27}
               \STATE STOP: $\AB$ is  positive definite  and $\nu_{i+1}$ is the definitizing shift;\label{line28}
            \ELSIF{$DEF(i)=-1$ and $A-\nu_{i+1}B\prec 0$}    \label{line29}
               \STATE STOP: $\AB$ is  negative definite  and $\nu_{i+1}$ is the definitizing shift; \label{line30}
            \ENDIF 
            \IF{$m=2$ or ($m=3$ and $i=0,1$)} \STATE 
                  $X_{i+1}=U_i Y_{i+1}$;\, $P_0=[\,]$;\, $P_1=X_1$;
            \ELSIF{$m=3$ and $i>1$} \STATE  $P_{i}=U_i^{(2)} Y_{i+1}^{(2)},\qquad X_{i+1}=U_i^{(1)} Y_{i+1}^{(1)}+P_{i}$; (see~\eqref{eq: Y}--\eqref{eq: P})
            \ENDIF 
              \IF{preconditioner is used}  \label{line32}
              \STATE form $R_{i+1}=AX_{i+1}-BX_{i+1}D_{i+1}$ and $W_{i+1}=T_{i+1}\, R_{i+1}$ by solving linear systems for the preconditioner $T_{i+1}=(A-\nu_{i+1}B)^{-1}$; \label{line33}
              \ELSE  
                 \STATE $W_{i+1}= R_{i+1}$;\label{line35}
              \ENDIF
              \STATE  
                   $U_{i+1}=[X_{i+1}, W_{i+1}]$ for {$m=2$}; 
             $U_{i+1}=[X_{i+1}, W_{i+1}, P_{i}]$ for $m=3$;\label{line43}
\ENDFOR 

\end{algorithmic}
\end{algorithm}

\begin{remark}\label{rem: subspace alg} Some remarks about Algorithm~\ref{alg:det def} are as follows:
\begin{itemize}
\item[(a)] The user chooses $m=2$ or $m=3$, which determines the dimension of the subspace, which is $p=m(k_++k_-)$.
\item[(b)] For a given matrix $U\in\C^{n\times p}$ we check the sufficient condition in the following way (Lines~\ref{line2}--\ref{line3}):
   \begin{algorithmic}[1]
 \FOR{$j=1$,~$2$, $\dotsc$, $p$}
     \STATE form $d_1(\tilde{u})=\tilde{u}^H A\tilde{u}$, $d_2(\tilde{u})=\tilde{u}^H B\tilde{u}$,
     where $u$ is the $j$th column of $U$ and $\tilde{u}=u/\|u\|_2$; 
     \IF {$d_1(\tilde{u})=d_2(\tilde{u})=0$} \label{line:suf1}
           \STATE STOP: $\AB$  is indefinite.
     \ELSIF {$\sqrt{d_1^2(\tilde{u})+d_2^2(\tilde{u})}<\mathsf{tol_{ind}}$}\label{line:suf2}
        \STATE STOP: $\AB$  is indefinite or close to an indefinite pair. 
      \ENDIF
 \ENDFOR 
\end{algorithmic}

  If $d_1(\tilde{u})=d_2(\tilde{u})=0$, then $\AB$ is indefinite.  If   $\sqrt{d_1^2(\tilde{u})+d_2^2(\tilde{u})}\approx 0$, we  terminate our algorithm. The analysis in section~\ref{subsec: num stab} shows that $\AB$ is  close to an indefinite pair.
\item[(c)] If the length of the current approximation interval  $\mathcal{I}_{i+1}:=(\nu_{i+1;1}^-,\nu_{i+1;1}^+)$ is very small (Lines~\ref{line23} and~\ref{line24}), this signals that in the case of the given definite $\AB$, the definiteness interval  is even smaller by Corollary~\ref{cor: interlacing}. 
Therefore, iterating further will not provide any more useful information if the tolerance  is approximately the maximum of the uncertainty in the 
input and   the unit   roundoff. So,  the algorithm terminates by declaring that the given matrix pair is  indefinite or close to an indefinite one (for the latter, see the analysis in section~\ref{subsec: num stab}).
  \item[(d)]   If $A-\nu_{i+1}B\succ 0$ ($A-\nu_{i+1}B\prec 0$) for some real $\nu_{i+1}$, then $\AB$ is positive (negative) definite by  definition (Lines~\ref{line27}--\ref{line30}). The positive (negative) definiteness of $A-\nu_{i+1}B$ is easily checked by computing Cholesky factorization, which is a numerically stable algorithm~\cite{HIGHAM1988103}. Notice that if the Cholesky factorization is done, then Algorithm~\ref{alg:det def} ends. Thus, we obtain only one complete Cholesky factorization, while in all previous iterations the factorization breaks down, possibly at an early stage.
  \item[(e)] Using indefinite preconditioners (Line~\ref{line32}), not necessarily in every iteration, can accelerate the detection of defi\-ni\-te\-ness, especially for $m=2$; see Example~\ref{ex: gen_hyper2} where definite matrix pairs are close to indefinite pairs. The numerical experiments show that the algorithms from~\cite{KressnerMPandurShao2013,MPandur2019} with one preconditioner converge much slower than with two preconditioners. Nevertheless, the aim of our algorithm is not to compute the exact boundaries of the definiteness interval, but rather to obtain a crude appro\-ximation of it, such that the midpoint of the approximate interval lies within the definiteness interval. Therefore, it suffices to use one preconditioner in our algorithm, if any is used.
  \item[(f)] The algorithm terminates when the iteration index  $i$
 reaches a maximum value, e.g., $i_{\max}=100$. In that case, we can run Algorithm~\ref{alg:det def} again  with another initial matrix of the same rank, or we can choose  subspaces with a larger dimension than in the previous case. The benefit of choosing a larger $p$ is demonstrated in the experiments in section~\ref{sec:exper}.\hfill{$\diamond$}
\end{itemize} 
\end{remark}

\subsubsection{Implementation issues}

We now give some remarks on the implementation of Algorithm~\ref{alg:det def}.

\textbf{Initial guess.} An initial guess~$U_0\in\C^{n\times k}$ is such that $\rank \,U_0=k$ and
$
\In\bigl(U_0^H BU_0\bigr) \geq (k_+,k_-,k_0) \ge (1,1,0),
$ where the inequalities are understood element-wise.
By induction from~\cite[Lemma 2.5]{KressnerMPandurShao2013} the inequality  $k_\pm \le p_\pm^{(i)}$  for $\In\big(U_i^H B U_i\big) =
(p_+^{(i)},p_-^{(i)},p_0^{(i)})$  holds for all iterations. 
 Choosing such an initial guess is quite straightforward when $B$ has a particular structure, such as in Example~\ref{ex: gen_hyper2}, or when $B$ is diagonal. If $B$ is not diagonal, but has at least one negative diagonal element, say $b_{jj}$, and at least one positive diagonal element, say $b_{ii}$, then $e_j$ and $e_i$ are $B$-negative and  $B$-positive vectors, respectively. So, we can form $U_0=[e_j,e_i]$ for $k=2$. Here, $e_i$ is the $i$-th column of $I_n$. 
If $A$  has at least one negative diagonal element,  and at least one positive diagonal element, then we can form an appropriate initial guess to apply Algorithm~\ref{alg:det def} to  $(B,A)$ which has the same eigenvectors as $(A,B)$ and its eigenvalues are reciprocal of the eigenvalues of $(A,B)$ (a zero eigenvalue corresponds to an infinite eigenvalue).
For general $B$,  we  perform a Hermitian indefinite decomposition 
\begin{equation}\label{eq: GJG}
    P^HBP=GJG^H,\quad J=\diag(j_{11},\ldots,j_{nn}),
\end{equation}
 where $P$ is a permutation matrix. If $B$ is nonsingular, then  $G$ is a lower block-triangular matrix with $1\times 1$ or $2\times 2$ diagonal blocks, and $J$ contains the inertia of $B$, i.e.\, $j_{ii}\in\{-1,1\}$, for $i=1\ldots,n$ (e.g.~\cite{BunchKaufmann1977,BunchMarcia2006,BunchParlett1971} and \cite{Slapnicar1998}). If $B$ is singular and $\rank(B)=r<n$, then $G$ from~\eqref{eq: GJG} is a lower block-trapezoidal $n\times r$ matrix of full column rank. In this case, $J$ is of order $r$ and contains  signs of nonzero eigenvalues of $B$. Then for $0\neq x\in \Cn$, and assuming $J=\diag(I_{n_+},-I_{n_-})$ we have
\begin{align}
x^HBx&=x^HGJG^Hx=(G^Hx)^HJ(G^Hx)=\qquad (y:=G^Hx)\nonumber\\
&=y^HJy=|y_1|^2+\ldots+|y_{n_+}|^2-|y_{n_++1}|^2-\ldots-|y_{n_++n_-}|^2.\label{eq:B poz neg stupci}
\end{align}
For any $0\neq y\in\mathbb{C}^{n_++n_-}$  such that $y^HJy>0$ in~\eqref{eq:B poz neg stupci}   solving the linear system $G^Hx=y$ gives a $B$-positive vector $x$. 
For example, if we want  $\In(U_0^HBU_0)=(1,1,0)$ we choose two linearly independent vectors $y$ and $z$ (for example, two columns of $I_n$) such that $y^HJy<0$  and $z^HJz>0$. After solving  linear systems $G^Hx_1=y$ and $G^Hx_2=z$, we get a $B$-negative vector $x_1$ and $B$-positive vector $x_2$ and form $U_0=[x_1,x_2]$ of full column rank. Since, 
$$\displaystyle \det U_0^HBU_0=\underbrace{x_1^HBx_1}_{<0}\underbrace{x_2^H Bx_2}_{>0}-\underbrace{|x_1^HBx_2|^2}_{\geq 0}\neq 0,$$
the matrix $U_0^HBU$ is nonsingular and indefinite.

 \textbf{Basis matrix.} Instead of using $X_{i-1}$ in~\eqref{eq: subspace basis} we use another matrix $P_{i-1}$ for numerical stability when $m=3$ because $X_{i-1}$ and $X_{i}$ tend to have linearly dependent columns as $i$ increases. For this purpose, the $3k\times k$ matrix $Y_{i+1}$ returned by the Rayleigh-Ritz procedure is partitioned as
  \begin{equation}\label{eq: Y}
      Y_{i+1} =
  \begin{bmatrix}
     Y_{i+1} ^{(1)}\\
      Y_{i+1} ^{(2)}
  \end{bmatrix},\quad  Y_{i+1} ^{(1)}\in\C^{k\times k},\quad Y_{i+1} ^{(2)}\in\C^{2k\times k}, 
  \end{equation}
  
and $U_i=[U_i^{(1)},U_i^{(2)}]$, $U_i^{(1)}\in\C^{n\times k}$, $U_i^{(2)}\in\C^{n\times 2k}$. Then the update
  \begin{equation}\label{eq: P}
  P_{i}=U_i^{(2)} Y_{i+1}^{(2)},\qquad X_{i+1}=U_i^{(1)} Y_{i+1}^{(1)}+P_{i},
  \end{equation}
yields a new basis  $U_{i+1}=[X_{i+1},W_{i+1},P_{i}]$  that in exact arithmetic spans the same space as the natural basis. Then we orthonormalize in the standard inner product on $\C^n$ or  $B$-orthonormalize $U_{i+1}$  in Line~\ref{line6} (for details see~\cite[Section~3]{KressnerMPandurShao2013,MPandur2019}).\footnote{We call the set $\{x_1,\ldots,x_p\}$ a $B$-orthonormal basis for the subspace $\mathcal{U}$ if it is a basis for  $\mathcal{U}$ and if $x_i$ and $x_j$ are $B$-orthonormal vectors for every $i\neq j$, where $i,j=1,\ldots,p$.} The results of performing Algorithm~\ref{alg:det def} with both types of orthonormalization are very similar in experiments from section~\ref{sec:exper}.

\textbf{Checking projected pairs.}
 We choose $p$ such that $p\ll n$, so $(A_i,B_i):=$ $(U_i^HAU_i,U_i^HBU_i)$ is a small pair with dense matrices.  In our algorithm, we need to detect whether a given projected pair is definite (Line~\ref{line7}), and if it is definite to compute eigenpairs around the definiteness interval (Line~\ref{line19}). To  fulfill this double duty, we apply the two-sided $J$-Jacobi method~\cite{SINGER20125704,Veselic1993} on 
 the equivalent Hermitian pair $(G_i^{-1}A_iG_i^{-H},J_i)$, where $G_iJ_iG_i^H$ is a Hermitian indefinite decomposition of the nonsingular matrix $B_i\neq \diag (\pm 1)$ such that $J_i=\diag(I_{p_+},I_{p_-})$ and $\In(B_i)=(p_+,p_-,0)$~\cite{Slapnicar1998}.  In case of  definite GEP $G_i^{-1}A_iG_i^{-H}y=\nu J_i y$,  the $J$-Jacobi method returns a diagonal matrix with eigenvalues of $(A_i,B_i)$ on its diagonal, and a corresponding matrix of eigenvectors $[\ldots y\ldots]$. The eigenvectors of  $(A_i,B_i)$ are then given by $x=G_i^{-H}y$. If $B_i=\diag(\pm 1)$, we directly apply the $J$-Jacobi method to $(A_i,B_i)$. If we $B$-orthonormalize $U_i$, in theory we have $B_i=\diag(\pm 1)$, but rounding errors will probably produce something instead of zero outside the diagonal of $B_i$. 
 For singular $B_i$ we use an auxiliary pair $(B_i,A_i-\mu B_i)$ with a nonsingular matrix $A_i-\mu B_i$~\cite[Remark 2.5]{MMPthesis}. 
 There are also some other possibilities for checking the definiteness of the projected pairs, but the method that uses $J$-Jacobi works very well in our numerical experiments. 



\section{Analysis of  Algorithm~\ref{alg:det def}}\label{subsec: analysis}
In this section, we analyze the operation count, convergence, and numerical stability of Algorithm~\ref{alg:det def}.

\subsection{On complexity and convergence analysis}
 The  maximum operation count of the $i$-th iteration  of Algorithm~\ref{alg:det def} for $i\geq 1$  applied to a Hermitian matrix pair of order $n$ with dense matrices is given in Table~\ref{tab:count}. For sparse matrices, the cost heavily depends on the sparsity structure of the matrices. For structured matrices, forming the preconditioned residual matrix can be done efficiently by direct solvers; e.g., using sparse LU. To reduce cost, one can also apply iterative methods for inexact solutions, such as GMRES. Checking definiteness of a small projected pair and the attempted Cholesky factorization on an indefinite matrix can terminate in early stages, leading to a low computational cost (see Example~\ref{exm:ispitivanje clement} from section~\ref{sec:exper}). 

\begin{table}[ht]
\footnotesize
\caption{The maximum operation count of $i$-th iteration of Algorithm~\ref{alg:det def}, $i\geq 1$, applied to a Hermitian matrix pair of order $n$ with dense matrices. Here,  $k=k_-+k_+\geq 2$, $p=2k$ for $m=2$, and $p=3k$ for $m=3$. We assume $p\ll n$.}
\label{tab:count}
\begin{center}
\begin{tabular}{@{}l|l@{}}
\toprule
 Step & Cost (flops)\\ \midrule
 Sufficient condition & $p(3n^2/2+3n/2-1)$\\
 Orthonormalize or  $B$-orthonormalize $U_i$ & $O(2p^2n)$ or $O(2pn^2)$\\
 Computing projected matrices $(U_i^HAU_i,U_i^HBU_i)$ & $p(4n^2-2n)+O(n)$  \\
 Checking definiteness of $(U_i^HAU_i,U_i^HBU_i)$ & $O(p^3)$\\
 Forming $X_{i+1}$ & $2kpn-kn$\\
 Cholesky for $A-\nu_{i+1}B$ & $O(n^3/3)$\\
 Forming $W_{i+1}$ & $O(2n^3/3)$ with LU or $O(2kt_{max}n^2)$ \\
  &  with maximal $t_{max}$ steps for GMRES\\
\bottomrule
\end{tabular}
\end{center}
\end{table}

 The following proposition ensures a monotonicity property  of Algorithm~\ref{alg:det def}. The key part is to use the computed Ritz vectors in a new search subspace. 
\begin{proposition}[Monotonicity]\label{le:mon.Ritz} If the $i$-th step, $i\geq 1$, of Algorithm~\ref{alg:det def} (for every $m\geq 2$) is executed and $DEF(i)=1$, 
then
\begin{align}
\nu_{i+1;j}^+\leq \nu_{i;j}^+\quad \textrm{for}\quad j=1,\ldots, k_+,\label{eq:monot +}\\
\nu_{i;j}^-\leq \nu_{i+1;j}^-\quad \textrm{for}\quad j=1,\ldots, k_-.\label{eq:monot -}
\end{align}
\end{proposition}
\begin{proof} Let $U_{i-1}=[X_{i-1},Z_{i-1}]$, $X_{i-1}\in\C^{n\times r}$, $Z_{i-1}\in\C^{n\times r(m-1)}$, $r=k_-+k_+$. Hence for $j=1,\ldots,k_\pm$, $(\nu_{i;j}^\pm,y_{i;j}^\pm)$ are inner  eigenpairs of definite $U_{i-1}^H\AB U_{i-1}$ and $X_i$ is the corresponding matrix of Ritz vectors. Therefore, $U_{i}=[X_{i},Z_{i}]$ and $\nu_{i+1;j}^\pm$ for $j=1,\ldots,k_\pm$ are inner  eigenvalues of definite  $(U_{i}^HAU_i,U_i^H B U_{i})$. For $U=[e_1,\ldots,e_{r}]$, the pair
$$U^H(U_i^HAU_i,U_i^HBU_i)U=(X_i^HAX_i,X_i^H B X_i)$$
is a projected  pair of definite pair $(U_i^HAU_i,U_i^HBU_i)$. Since $\nu_{i;j}^\pm$ are inner  eigenvalues of  definite  $(X_i^HAX_i,X_i^H B X_i)$, by using the Cauchy-type interlacing inequalities we obtain~\eqref{eq:monot +} and \eqref{eq:monot -}.
\end{proof}

 Therefore, when we apply  Algorithm~\ref{alg:det def} to any matrix pair, definite or not, the intervals $(\nu_{i+1;1}^-,\nu_{i+1;1}^+)$ shrink as $i$ increases. Moreover, if Algorithm~\ref{alg:det def} is used with parameters $k=2$ and $k_\pm=1$ (the so-called “vector version”) on a positive definite matrix pair, and the procedure is continued even after definiteness has been detected, then the sequences $\{\nu_{i+1;1}^+\}$ and $\{\nu_{i+1;1}^-\}$ will converge to certain eigenvalues of the given matrix pair, as established in~\cite[Theorem~2.1]{BennerLiang2016}. The sequences  $\{\nu_{i+1;1}^+\}$ and  $\{\nu_{i+1;1}^-\}$ converge to some $B$-positive and $B$-negative eigenvalue, respectively, by Theorem~\ref{tm:interlacing}.  But convergence to the extreme  eigenvalues $\lambda_1^\pm$ is guaranteed only if $\lambda_1^+\leq\nu_{i+1;1}^+<\lambda_2^+$ for some $i$ and  $\lambda_2^-<\nu_{i'+1;1}^-\leq \lambda_1^-$ for some $i'$. 
Therefore, the sequence of definiteness intervals associated with the definite projected pairs converges to an interval that contains the definiteness interval of the original definite matrix pair. Consequently, we can expect Algorithm~\ref{alg:det def} to terminate very quickly when the definiteness interval of the original matrix pair is relatively large.
 
 After detecting definiteness, the final output of Algorithm~\ref{alg:det def} (a definitizing shift $\lambda_0$, an interval $(\nu_{i+1;1}^-,\nu_{i+1;1}^+)$, a matrix $X_{i+1}$) could be used as an input of the algorithms from~\cite{KressnerMPandurShao2013,MPandur2019}:
\begin{itemize}
\item[(a)] a definitizing shift is used to form one positive definite preconditioner in  algorithms from~\cite[Algorithm 1]{KressnerMPandurShao2013},~\cite{MPandur2019};
\item[(b)] since the final interval $(\nu_{i+1;1}^-,\nu_{i+1;1}^+)$ is a crude approximation of the definiteness interval of a given definite matrix pair, we could pick two scalars near the boundaries of $(\nu_{i+1;1}^-,\nu_{i+1;1}^+)$ and use them to form two possibly excellent preconditioners in  algorithms from~\cite[Algorithm 2]{KressnerMPandurShao2013},~\cite{MPandur2019};
\item[(c)] the final computed Ritz vectors given as the columns of a  matrix $X_{i+1}$  could serve as an initial input in algorithms from~\cite{KressnerMPandurShao2013,MPandur2019}.
\end{itemize}
Convergence rate estimates for the sequences $\{\nu_{i;1}^\pm\}$ generated by the algorithms of~\cite{KressnerMPandurShao2013,MPandur2019} with parameters $k=2$ and $k_\pm=1$ can be found in~\cite[Theorems~4 and~5]{BennerLiang2016}. See also the sharp convergence results provided in~\cite[Theorems~4.2 and~4.3]{MPandur2019}.
\subsection{A modification of Algorithm~\ref{alg:det def}}\label{rem: neg curv}
 If the Cholesky factorization of $A-\nu_{i+1} B$ breaks down, we can compute a nonzero vector $z$ such that \begin{equation}\label{eq: curv}
    z^H(A-\nu_{i+1} B)z\leq 0
\end{equation}  
using the partial factorization~\cite[Section 3]{Tisseur2010}. Such a vector  $z$ is called a direction of negative  curvature. It can help to further reduce the length of the current interval $(\nu_l,\nu_u):=\mathcal{I}_{i+1}$ in Algorithm~\ref{alg:det def} for possibly positive definite $(A,B)$. Following~\cite{More1993}, we show how to update $(\nu_l,\nu_u)$ given $\nu_{i+1}\in (\nu_l,\nu_u)$:
\begin{itemize}
\item[(a)] if $z^HBz<0$, then $$z^HAz-\nu z^HBz< z^HAz-\nu_{i+1} z^HBz\leq 0,\quad \textrm{for}\,\,\nu<\nu_{i+1}$$
and therefore $\nu_l:=\nu_{i+1}$;
\item[(b)] if $z^HBz>0$, then $$z^HAz-\nu z^HBz< z^HAz-\nu_{i+1} z^HBz\leq 0,\quad \textrm{for}\,\,\nu>\nu_{i+1}$$
and therefore $\nu_u:=\nu_{i+1}$.
\end{itemize}
If the indefiniteness is not detected in the next iteration,  we intersect this updated interval with the definiteness interval   $\mathcal{I}_{i+2}=(\nu_{i+2;1}^-,\nu_{i+2;1}^+)$ from the next iteration, and choose $\nu_{i+2}$ from that intersection. If we choose each $\nu_{i+1}$ as the middle of the interval, then Algorithm~\ref{alg:det def} with this modification in updating the intervals is a bisection like algorithm, so it converges at least linearly.

Therefore, Algorithm~\ref{alg:det def} detects definiteness or (near) indefiniteness with\-in a finite number of iterations.  We do not write this modification in Algorithm~\ref{alg:det def} as a new algorithm for the sake of brevity.

 A direction of negative curvature $z$, the value of $z^H(A-\nu_{i+1}B)z$, the effects of rounding errors (the correctness of the inequality~\eqref{eq: curv} for a computed $z$ in floating point arithmetic) depend on the way how the attempted Cholesky factorization is made, without  pivoting or with some kind of pivoting, e.g.,~\cite[equations (3.3) or (3.5)]{Tisseur2010}.  The attempted Cholesky factorization of  an indefinite matrix may terminate faster when using pivoting than without pivoting,  resulting in a smaller numerical cost (so called ``early-exit complete pivoting''~\cite{Tisseur2010}).   

\subsection{Numerical stability}\label{subsec: num stab}
 Algorithm~\ref{alg:det def} uses multiple determination criteria.  The following theorems give the determination criteria under which our algorithm is backward stable.

 \begin{theorem}
 Whenever a determination of  Algorithm~\ref{alg:det def} is made  by finding a vector $z$ such that $\|z\|_2=1$ and $z^HAz=z^HBz=0$ (Lines~\ref{line2}-\ref{line3}),  or when the Cholesky factorization is successfully completed (Lines~\ref{line27}-\ref{line28} or~\ref{line29}-\ref{line30}), it is numerically stable in the sense that it is correct for a Hermitian perturbed pair $(A+\Delta A, B+\Delta B)$ with
\begin{equation}\label{eq: perturbed pair}
    \|[\Delta A\,\,\Delta B]\|_2\leq c_n\, \varepsilon \|[A\,\,B]\|_2,
\end{equation}
where $c_n$ is a modest constant, and $\varepsilon$ is the unit roundoff. 
 \end{theorem}
 \begin{proof}   The proof is given in~\cite{Tisseur2010} using backward error analysis. More precisely, if we test the definiteness of $$A\sin t_i  +B\cos t_i :=A\frac{1}{\sqrt{1+\nu_i^2}}+B\frac{-\nu_i}{\sqrt{1+\nu_i^2}}$$
instead of $A-\nu_i B$ using the Cholesky factorization, then we can directly  use the argument from~\cite[Section 2]{Tisseur2010}.
 \end{proof}
 \begin{theorem}
 When  a determination of  Algorithm~\ref{alg:det def} with standard orthonormalization is made by finding some indefinite projected pair (Lines~\ref{line8}-\ref{line9}), it is numerically stable in the sense that it is correct for a Hermitian perturbed pair $(A+\Omega A, B+\Omega B)$ with
\begin{equation}\label{eq: perturbed pair2}
   \max\{ \|\Omega A\|_2,\|\Omega B\|_2\}\leq c_p \,\varepsilon  \max\{ \|A\|_2,\|B\|_2\},
\end{equation}
where $c_p$ is a modest constant, and $\varepsilon$ is the unit roundoff.    
 \end{theorem}
 \begin{proof}
   Assume that a projected perturbed pair $$(U^HAU+E,U^HBU+F),\,\textrm{where} \,\,E=E^H,\, F=F^H \,\textrm{and}\,\, U^HU=I_p,$$ is declared to be indefinite by some numerically stable algorithm such that
\begin{equation*}\label{eq: perturbed pair3}
    \max\{ \|E\|_2,\|F\|_2\}\leq c_p\, \varepsilon  \max\{ \|U^HAU\|_2,\|U^HBU\|_2\},
\end{equation*}
where $c_p$ is a modest constant, and $\varepsilon$ is the unit roundoff. Let $\mathcal{U}=\mathrm{diag}(U,U)$, so $\|\mathcal{U}\|_2=\|\mathcal{U}^H\|_2=1$.
For Hermitian perturbation matrices $\Omega A := UEU^H$ and $\Omega B := UFU^H$, the following holds
\begin{align*}
    \max\{ \|\Omega A\|_2,\|\Omega B\|_2\} &= \Big\|
    \begin{bmatrix}
        \Omega A& 0\\
        0& \Omega B
    \end{bmatrix}\Big \|_2=\Big\|
    \begin{bmatrix}
        UEU^H& 0\\
        0& UFU^H
    \end{bmatrix}\Big \|_2\\
    &= \Big\|\mathcal{U} \begin{bmatrix}
        E& 0\\
        0& F
    \end{bmatrix}\mathcal{U}^H\Big \|_2\leq \|\mathcal{U}\|_2 \Big\| \begin{bmatrix}
        E& 0\\
        0& F
    \end{bmatrix}\Big \|_2\|\mathcal{U}^H\|_2\\
    &= \Big\| \begin{bmatrix}
        E& 0\\
        0& F
    \end{bmatrix}\Big \|_2=\max\{ \|E\|_2,\|F\|_2\}\\
    &\leq c_p\, \varepsilon  \max\{ \|U^HAU\|_2,\|U^HBU\|_2\}\\
    &=c_p\, \varepsilon \Big\|
    \begin{bmatrix}
        U^HAU& 0\\
        0& U^HBU
    \end{bmatrix}\Big \|_2=c_p\, \varepsilon \Big\|
    \mathcal{U}^H\begin{bmatrix}
        A& 0\\
        0& B
    \end{bmatrix}\mathcal{U}\Big \|_2\\
    &\leq c_p\, \varepsilon \Big\|
    \begin{bmatrix}
        A& 0\\
        0& B
    \end{bmatrix}\Big \|_2= c_p\, \varepsilon  \max\{ \|A\|_2,\|B\|_2\}.
\end{align*}

 \end{proof}

\begin{theorem} Let $(A,B)$ be a definite matrix pair.
 Whenever a determination of  Algorithm~\ref{alg:det def} is made  by
 \begin{itemize}
     \item[a)]  finding a unit vector $u$ such that $\sqrt{(u^H Au)^2+(u^H Bu)^2}\leq \mathsf{tol_{ind}}$ (Lines~\ref{line2}-\ref{line3}), or 
     \item[b)]   the length $|\nu_{i+1;1}^+-\nu_{i+1;1}^-|$ of the definiteness interval  of some definite projected pair is smaller than  $\mathsf{tol}$ (Lines~\ref{line23}-\ref{line24}), such that $\|X^{-1}\|_2^2\ll \mathsf{tol}^{-1}$ holds for the eigenvector matrix $X$ of $(A,B)$~\cite{LiBai2013}[Lemma~3.8, item~1.],
 \end{itemize}
 where    $\mathsf{tol_{ind}}$ and  $\mathsf{tol}$  are positive tolerances, then the given matrix pair is close to an indefinite pair. More precisely, that matrix pair is within the distance $\mathsf{tol_{ind}}$ for the case a), and  $\frac{\mathsf{tol}}{2}\|X^{-1}\|_2^2$ for the case b), of an indefinite pair. In these cases, small perturbations in matrices $A$ and $B$ can cause the loss of definiteness. 
 \end{theorem}
 \begin{proof}
If the given matrix pair is deemed to be indefinite or definite but close to an indefinite pair on Lines~\ref{line2}-\ref{line3} (that is, Line~\ref{line:suf2}), we justify the decision as follows. 
The Crawford number $\gamma(A,B)$~\eqref{eq: Crawford}
equals the distance $d(A,B)$ from a Hermitian  $\AB$ (definite or indefinite) to the nearest indefinite pair~\cite[Theorem 2.2]{htvd2002}, which is defined by
\begin{equation}\label{eq: Craw dist}
    d(A,B)=\min\{\|[\Delta A\,\,\Delta B]\|_2: z^H(A+\Delta A+i(B+\Delta B))z=0, \,\text{some}\, z\neq 0\}.
\end{equation}
Suppose  $(A,B)$ is actually definite and that there exists $u\in\mathbb{C}^n$ such that $\|u\|_2=1$, $u^H Au\approx 0$ and $u^H Bu\approx 0$, that is
\begin{equation*}\label{eq: close ind}
    \sqrt{(u^H Au)^2+(u^H Bu)^2}\leq \mathsf{tol_{ind}}.
\end{equation*}
Therefore,
 $d(A,B)=\gamma(A,B)\leq \mathsf{tol_{ind}}$,
 which indicates that $(A,B)$ lies within a distance $\mathsf{tol_{ind}}$ of an indefinite pair.

 If the given matrix pair is deemed to be indefinite or definite but close to an indefinite pair on Lines~\ref{line23}-\ref{line24}, we justify the decision as follows. 
 Let $(A,B)$ be a positive definite pair with indefinite $B$ such that $\In(B)=(n_+,n_-,n_0)$. Corollary~\ref{cor: interlacing} implies $\lambda_1^+-\lambda_1^-\leq \nu_1^+-\nu_1^-$,
where $\lambda_1^\pm$ and $\nu_1^\pm$ are eigenvalues of $(A,B)$ and $(U^HAU,U^HBU)$, respectively, for any $U$ with full column rank such that $U^HBU$ is indefinite. 
When  a determination of  Algorithm~\ref{alg:det def} is made by finding some positive definite projected pair $(U^H_{i}AU_i, U^H_{i}BU_i)$ such that
$\nu_{i+1;1}^+-\nu_{i+1;1}^{-}<\mathsf{tol}$, this implies that
\begin{equation}\label{eq: len}
    \mathrm{len}:=\lambda_1^+-\lambda_1^-<\mathsf{tol}.
\end{equation}
Recall that there exists a nonsingular matrix $X$ such that (e.g,~\cite{LiBai2013}[Lemma~3.8, item~1.])
\begin{equation}\label{eq: X diag}
 X^HAX=\mathrm{diag}({\Lambda_+,-\Lambda_-},I_{n_0})=:D,\quad X^HBX=\mathrm{diag}(I_{n_+},-I_{n_-},0_{n_0})=:J,   
\end{equation}
are both diagonal, where $\Lambda_+=\mathrm{diag}(\lambda_{n_+}^+,\ldots,\lambda_1^+)$   and $\Lambda_-=\mathrm{diag}(\lambda_{1}^-,\ldots,\lambda_{n_-}^-)$, with $\lambda_{i}^\pm$ as in~\eqref{eq:eigenvalueordering}. The representations in~\eqref{eq: X diag} are uniquely determined by $(A,B)$, up to a simultaneous permutation
of the corresponding diagonal elements. Let $\lambda_0=(\lambda_1^++\lambda_1^-)/2$, $\Delta B=0$ and $\Delta A$ be a Hermitian perturbation  of the matrix $A$ such that
$$\begin{array}{rl}
&X^H(A+\Delta A)X\\
&=\mathrm{diag}(\lambda_{n_+}^+,\ldots,\lambda_2^+,\lambda_1^+-\mathrm{len}/2,-(\lambda_{1}^-+\mathrm{len}/2),-\lambda_2^-,\,\ldots,-\lambda_{n_-}^-,1\ldots,1)\\
&=\mathrm{diag}(\lambda_{n_+}^+,\ldots,\lambda_2^+,\lambda_0,-\lambda_0,-\lambda_2^-,\ldots,-\lambda_{n_-}^-,1\ldots,1)=: \hat{D}.
\end{array}
$$
The diagonal matrix $\hat{D}-\lambda_0J$ has non-negative eigenvalues, therefore, it is a positive semidefinite matrix. Consequently, the matrix $A+\Delta A-\lambda_0 B$ is also positive semidefinite because congruence preserves semidefiniteness.
 Therefore, the Hermitian perturbed pair $(A+\Delta A, B)$ is a  positive semidefinite\footnote{Let $A,B\in \Cnn$. 
  A Hermitian  matrix pair $(A,B)$ is called \emph{positive} \emph{(negative)} \emph{semidefinite} if there exists a real $\lambda_0$ such that the matrix $A-\lambda_0 B$ is positive (negative) semidefinite. We have the following eigenvalue ordering for the finite eigenvalues of a positive semidefinite $(A,B)$: $\lambda_{n_-}^{-} \leq\dotsb \leq \lambda_{1}^{-} \leq \lambda_{1}^{+} \leq
\dotsb \leq \lambda_{n_+}^{+}$, where  $\In(B)=(n_+,n_-,n_0)$~\cite{Lancaster2005,LiBai2013}. Also, $\{\lambda\in\R: A-\lambda B\succeq 0\}=[\lambda_1^-,\lambda_1^+]$.}  matrix pair such that $\{\lambda\in\R: A+\Delta A-\lambda B\succeq 0\}=\{\lambda_0\}$ is a singleton since $\lambda_0$ is a multiple eigenvalue that is of mixed type (that is, it has a $B$-positive and a $B$-negative eigenvector) (\cite{LiBai2013}[Lemma~3.8, items~2.-3.]). 
Therefore, $(A+\Delta A, B)$ is a Hermitian indefinite pair, as no open definiteness interval exists. Since $(A+\Delta A, B)$ is indefinite, there exists a unit vector~$u$ such that $u^H(A+\Delta A+iB)u=0$ by~\eqref{eq: Crawford}.
So, for our choice of the Hermitian perturbation $$(\Delta A,\Delta B)=(X^{-H}(\hat{D}-D)X^{-1}, 0),$$ by~\eqref{eq: Craw dist} and~\eqref{eq: len} we have 
\begin{equation*} \label{eq: gamma tol}
\begin{split}
    \gamma(A,B)&\leq \|[\Delta A\,\,\Delta B]\|_2\leq \|X^{-H}\|_2 \|\hat{D}-D\|_2  \|X^{-1}\|_2\\
    &=\frac{\mathrm{len}}{2} \|X^{-1}\|_2^2<\frac{\mathsf{tol}}{2}\|X^{-1}\|_2^2.
\end{split}
\end{equation*}
Therefore, our definite pair $(A,B)$ is within the distance $\frac{\mathsf{tol}}{2}\|X^{-1}\|_2^2$ of an indefinite pair. 
If  $\|X^{-1}\|_2^2$ is not too large ($\|X^{-1}\|_2^2\ll \mathsf{tol}^{-1}$), and $\mathsf{tol}$ is approximately the unit roundoff, then $(A,B)$ is close to an indefinite pair.

We conclude that Algorithm~\ref{alg:det def} can wrongly diagnose indefiniteness only when a definite $(A,B)$ is close to an indefinite pair.
 \end{proof}

\section{Numerical experiments}\label{sec:exper}

In this section, we give several numerical experiments to demonstrate the behavior of Algorithm~\ref{alg:det def}, with and without preconditioners~\cite{MMPcode}.
For forming exact preconditioners, we use the MATLAB backslash operator to solve linear systems. For forming inexact preconditioners we use GMRES with a maximum of 15 iterations. The results are presented for Algorithm~\ref{alg:det def} as described in section \ref{sec:det def subspaces} without the computation of a direction of negative curvature $z$ from \eqref{eq: curv} (the overall number of iterations is similar, but it is more time consuming when $z$ is computed). 
We use tolerance $\mathsf{tol}= 10^{-12}$ in Line~\ref{line23} of Algorithm~\ref{alg:det def}   and $\mathsf{tol_{ind}}= 10^{-4}$ in Line~\ref{line:suf2} of Remark~\ref{rem: subspace alg}~(b) (unless otherwise stated). 
All experiments have been performed in MATLAB R2020a on Processor	Intel(R) Core(TM) i7-1065G7 CPU  1.30GHz,  4 Core(s), 8 Logical Processor(s), 16 GB RAM.

 We compare our algorithm with the improved arc algorithm~\cite[Algorithm~2.3]{Tisseur2010} since it uses the attempted Cholesky factorization in each iteration.  
 For Cholesky factorization in the arc algorithm, we use a handwritten function for dense matrices (different types of pivoting can be used) and the MATLAB function $\mathtt{chol}$ for sparse matrices. For dense matrices, we run the arc algorithm   without pivoting and with complete pivoting (in which at each stage the pivot is chosen as the largest diagonal element in the active part of the matrix). We use $[1,1,\ldots,1]^T/\sqrt{n}$ as an initial vector in the arc algorithm.

We also compare our algorithm with the algorithms from ~\cite{Ding2018,PandurKuzmano}.
 For computing the smallest eigenvalue, and its corresponding eigenvector in~\cite[Algorithm~4.2 and 4.3]{PandurKuzmano},  we use   LOBPCG~\cite{Knyazev2001},\footnote{MATLAB code available from \url{https://mathworks.com/matlabcentral/fileexchange/48-lobpcg-m}.} with GMRES preconditioners~\footnote{For efficiency of~\cite[Algorithm~4.1]{PandurKuzmano} we use the eigenvectors from the last iteration to initialize each call to LOBPCG. The preconditioner is of the form  $T_{k+1}=(A \cos\theta_{k+1}+B\sin\theta_{k+1}-c_{k+1}I)^{-1}$, where $c_{k+1}=\lambda_{min}(V_k^H(A\cos\theta_{k+1}+B\sin\theta_{k+1}) V_k)$, available at the $k-$th iteration of~\cite[Algorithm~4.2]{PandurKuzmano}, where $V_k$ is the basis matrix of the current subspace.} and the block size of two, or MATLAB function $\mathsf{eigs}$.  LOBPCG and   $\mathsf{eigs}$, are stopped when the maximum number of 100 iterations is achieved or the norm of the residual is less than $10^{-8}$. 
We use the block size of two,  GMRES preconditioners and stop ~\cite[Algorithm~5.1]{Ding2018} when the norm of the residual is smaller than $\mathsf{tol}=10^{-8}$. Additionally, we stop the algorithms in~\cite{Ding2018,PandurKuzmano} when the upper bound for the Crawford number is less than $10^{-4}$ indicating that the original matrix pair is close to an indefinite one. 

\begin{example}\label{exm:ispitivanje clement} (\textbf{Tridiagonal and diagonal matrices})
{\rm 
In this experiment we consider the indefinite pairs $(H,J_r)$ of order $n=500$,
where $H$ is a tridiagonal real symmetric matrix   obtained  using the MATLAB code:
\begin{verbatim}
    H=gallery('clement',n,1);
    Hmax=max(max(abs(H)));
    H=H./Hmax;
\end{verbatim}
and $J_r:=\diag(I_{n-r},-I_{r}),\,r=10,400$.
 We present the results of applying the arc algorithm~\cite[Algorithm~2.3]{Tisseur2010} and Algorithm~\ref{alg:det def}  with $m=2$,  with and without the preconditioners, and two initial matrices:
 \begin{center}
$\begin{array}{cl}
U_{0;1} & =[I_{500}(:,1),-I_{500}(:,500)],\\
U_{0;2} & =[[1,\ldots,1,\underbrace{0,\ldots,0}_{r\,\textrm{times}}]^T,[0,\ldots,0,\underbrace{-1,\ldots,-1}_{r\,\textrm{times}}]^T].
\end{array}$
\end{center}
The results are given in Table~\ref{tab:ispitivanje clement indef}. When the first projected pair is indefinite, noted by ``Stage: --'', there was no attempted Cholesky factorization. Both algorithms quickly detect the indefiniteness of $(H,J_r)$. Algorithm~\ref{alg:det def} has very low computational cost in this example. We use a symmetric indefinite decomposition  $GJG^T$ from~\cite[Section 2]{Slapnicar1998}, but with pivoting strategy designed for tridiagonal matrices from~\cite{BunchMarcia2006} in our algorithm to form an exact preconditioned residual.

\begin{table}[ht]
\footnotesize
\caption{Comparison of the arc algorithm~\cite[Algorithm~2.3]{Tisseur2010} and Algorithm~\ref{alg:det def} with~$m=2$ with and without preconditioners  in detecting indefinite pairs from  Example~\ref{exm:ispitivanje clement}. Stage denotes on which diagonal element Cholesky factorization without pivoting terminates. For Algorithm~\ref{alg:det def} and fixed $r$  upper results are for $U_{0;1}$ and lower results are for $U_ {0;2}$.}
\label{tab:ispitivanje clement indef}
\begin{center}
\begin{tabular}{@{}l|l|l|l@{}}
\toprule
 & \multirow{2}{*}{\cite[Alg.~2.3]{Tisseur2010}} & \multicolumn{2}{c}{Alg.~\ref{alg:det def} with $m=2$}\\
 & & without preconditioners & with preconditioners \\
  \midrule
 \multirow{4}{*}{$r=10$} & Stage: & Stage: -- & Stage: --\\
                          & 18, 491; &  1. projected pair is indefinite  & 1. projected pair is indefinite\\
                           \cline{3-4}
                          & arc length $>\pi$& Stage: 1, 1; & Stage: 1, 1, 3, 5, 6;\\
                          &  & 3. projected pair is indefinite & 6. projected pair is indefinite\\
                            \hline
 \multirow{4}{*}{$r=400$} & Stage: & Stage: -- & Stage: --\\
                          & 1, 21,67,101; & 1. projected pair is indefinite & 1. projected pair is indefinite\\
                          \cline{3-4}
                          &arc length $>\pi$  & Stage:  6, 3; & Stage: 6, 1, 1;\\
                          &  & 3. projected pair is indefinite & 4. projected pair is indefinite\\
\bottomrule
\end{tabular}
\end{center}
\end{table}
}
\end{example}

\begin{example}\label{ex: gen_hyper2} (\textbf{Quadratic eigenvalue problem})
The quadratic eigenvalue pro\-blem (QEP)
\begin{equation}\label{eq:qep}
\mathbf{Q}(\lambda)x:=(\lambda^2M+\lambda D+K)x=0,
\end{equation}
where $\lambda\in\C$ is called an eigenvalue, and $x\in\C^n$ is called an eigenvector, is said to be hyperbolic if  $M,D,K\in\C^{n\times n}$ are Hermitian, $M$ is positive definite, as well as 
$(z^HDz)^2-4(z^HMz)(z^HKz)>0$ for all nonzero $z\in \Cn$. 
 It is well known~\cite{htvd2002,Veselic1993} that hyperbolicity of the  QEP in~\eqref{eq:qep} is equivalent to the  definiteness of the matrix pair $(A,B)$ with
 \begin{equation}\label{eq:pair A B}
A=\begin{bmatrix}
M & 0\\
0 & -K
\end{bmatrix},\quad
B=\begin{bmatrix}
0 & M\\
M & D
\end{bmatrix}.
\end{equation}
They have the same eigenvalues, and the eigenvectors are easily connected.



We construct a set of 30 hyperbolic quadratics of order $n=500$ using the MATLAB function $\mathsf{gen\_hyper2}$ in the collection NLEVP~\cite{nlevp} with  the following  eigenvalue distribution (the other parameters are chosen randomly):
\begin{itemize}
  \item [] \textbf{type 1}: $\lambda_k^{\pm}$, $k=1,\ldots,n$, is uniformly distributed in $[-500,-250]\cup[-240,-1]$. The mean value of the  lengths of the computed definiteness intervals is $10.5$.  
\end{itemize}
Similarly, we construct a set of $20$ hyperbolic quadratics of order $n=500$ with  the following  eigenvalue distribution, as in~\cite{Voss2010}:
\begin{itemize}
  \item [] \textbf{type 2}: $\lambda_k^{\pm}$, $k=1,\ldots,n$, is uniformly distributed in $[-100,-1]$,  and the length of the definiteness interval is $10^{-j}$ for $j=1,\ldots, 20$. So, the last definiteness interval is in the theory of length $10^{-20}$, but in floating point arithmetic with $\varepsilon=2^{-53}\approx 1.1\times 10^{-16}$, the unit roundoff, we have no definiteness interval: since $\lambda_1^-=\lambda_1^+$. Therefore, for $j=15,\ldots,20$, we have positive semidefinite matrix pairs with no definiteness intervals.
  \end{itemize}
The corresponding matrices in the constructed quadratics are real and dense. We run Algorithm~\ref{alg:det def} with $m=2$ and with $m=3$  on the matrix pair from~\eqref{eq:pair A B} of order $2n=1000$.\footnote{We emphasize that there exist algorithms for detecting a hyperbolic QEP that work directly with matrices $M$, $D$, $K$, which reduces storage requirements and the cost per iteration, such as~\cite{Tisseur2010,Voss2010,MPandur2020}. 
Here we just want to exploit the behavior of Algorithm~\ref{alg:det def} on the linearization matrix pair which has a special structure.} An initial matrix $U_0$ is chosen such that  $k_+$ $B$-positive columns are from $[0_n;I_n]$ and $k_-$ $B$-negative columns are from $[M^{-1}D;-I_n]$, since $D$ is positive definite. We use different dimensions $p$ of the subspaces for $m=3$: for $k_\pm=1$ we have $p=6$, for $k_\pm=2$ we have $p=12$, and for $k_\pm=3$ we have $p=18$. We run four versions of Algorithm~\ref{alg:det def} with $m=3$:
\begin{itemize}
\item[(a)] without preconditioning;
    \item[(b)] with exact preconditioning in every fifth iteration, in other iterations we do not use preconditioning;
    \item[(c)]  with exact preconditioning in every  iteration;
     \item[(d)] with inexact preconditioners in every  iteration.
\end{itemize}
Similarly, for $m=2$ we have $p=4,8,12$ and four versions depending on the choice of a preconditioning.

\begin{table}[ht]
\footnotesize
\caption{The results for  detecting 30 definite pairs from~\eqref{eq:pair A B} of order $2n=1000$ from  Example~\ref{ex: gen_hyper2} for \textbf{type 1}. The minimum, average and maximum number of iterations, and the average CPU time are given for all algorithms.}
\label{tab:ex: gen_hyper2 type 1 n=500}
\begin{center}
\begin{tabular}{@{}l|rrrr|rrrr@{}}\toprule
Algorithm & $\min$ & $\mathrm{mean}$ & $\max$ & CPU time &  $\min$ & $\mathrm{mean}$ & $\max$ & CPU time\\
\midrule
 & \multicolumn{4}{c|}{without precond.} & \multicolumn{4}{c}{exact precond. in every 5th iter. }\\
Alg.~\ref{alg:det def}, $m=2$  & 1 & 13 & $16$ & $0.20$ &    1 & 11 & 16 & $0.22$ \\
Alg.~\ref{alg:det def},  $m=3$ & 1 & 4 & $4$ & $0.07$  &  1 & 4 & 4 & $0.07$ \\
\midrule
 & \multicolumn{4}{c|}{exact precond.} & \multicolumn{4}{c}{GMRES precond.}\\
Alg.~\ref{alg:det def}, $m=2$  & 1 & 4 & 5 & $0.16$ &  1 & 5 & $14$ & $0.69$\\
Alg.~\ref{alg:det def}, $m=3$  & 1 & 4 & 5 & $0.15$  &  1 & 4 & 12 & $0.48$ \\
\midrule
 \multirow{2}{*}{\cite[Alg.~2.3]{Tisseur2010}}& \multicolumn{4}{c|}{without pivoting} & \multicolumn{4}{c}{with complete pivoting}\\
 & 6 & 6 & 7 & $9.54$ &  4 &6 & 7 & $12.08$ \\ \midrule 
 \cite[Alg.~5.1]{Ding2018}  & 4 & 11 & 21 & $0.45$  &&&&\\ \midrule 
 & \multicolumn{4}{c|}{LOBPCG} & \multicolumn{4}{c}{$\mathsf{eigs}$}\\ 
  \cite[Alg.~4.2]{PandurKuzmano} & 2 & 2 & 3 & $11.67$  &2&2&3& $8.93$\\ 
    \cite[Alg.~4.3]{PandurKuzmano} & 2 & 2 & 2 & $11.19$  &2&2&3& $8.85$\\
 \bottomrule
\end{tabular}
\end{center}
\end{table}

 When $\mathsf{tol_{ind}}=10^{-4}$, the algorithm ends immediately in Line~\ref{line2} for an initial matrix $U_0$ declaring  in all cases,  that the given matrix pair is close to an indefinite matrix pair. Therefore, to have a better insight into the performance of Algorithm~\ref{alg:det def}, we use smaller tolerance $\mathsf{tol_{ind}}=10^{-12}$ and present the results when $B$-orthonormalization is used in Line~\ref{line6}. The results for \textbf{type 1} are given in Table~\ref{tab:ex: gen_hyper2 type 1 n=500} only for $k_\pm=1$ (the results for other $k_\pm$ are very similar and are omitted).  In all cases  Algorithm~\ref{alg:det def}  detects the definiteness in less than 17 iterations. Algorithm~\ref{alg:det def} with $m=3$  is more efficient than with $m=2$, especially without preconditioning.

 The results  for running~\cite[Algorithm~2.3]{Tisseur2010}   are given in the middle of Table~\ref{tab:ex: gen_hyper2 type 1 n=500}. The algorithm with complete pivoting achieves better performance. Although the overall number of iterations of the arc algorithm is comparable to that of Algorithm~\ref{alg:det def} with exact preconditioners, the arc algorithm requires more computational time.

 The results for running~\cite[Algorithm~5.1]{Ding2018} and \cite[Algorithm~4.2, Algorithm~4.3]{PandurKuzmano} are at the  bottom of Table~\ref{tab:ex: gen_hyper2 type 1 n=500}. In all cases, these algorithms stop when the approximation of the Crawford number, which is of order $10^{-6}$ in all cases, is less than $10^{-4}$. The algorithms presented in~\cite{PandurKuzmano} are the most time-consuming ones, as in certain cases they require the full limit of 100 steps to compute the eigenpair.


\begin{table}[ht]
\footnotesize
\caption{The results for  detecting 14 definite pairs and 6 indefinite pairs from~\eqref{eq:pair A B} of order $2n=1000$ from  Example~\ref{ex: gen_hyper2} for \textbf{type 2}. The minimum, average and maximum number of iterations, and the average CPU time are given for all algorithms. The results are separately given for~\cite[Algorithm~2.3]{Tisseur2010} without pivoting (noted by ``\textit{no piv.}'') and with complete pivoting (noted by ``\textit{compl. piv.}'') in the attempted Cholesky factorizations. 
}
\label{tab:ex: gen_hyper2 type small n=500}
\begin{center}
\begin{tabular}{@{}l|r|rrrrr@{}}
\toprule
Algorithm &$p$ & $\min$ & $\mathrm{mean}$ & $\max$ & CPU time  \\
\midrule
\multirow{3}{*}{Alg.~\ref{alg:det def}, $m=2$} &4  & 7 & 10 & 14 & $0.42$  \\
& 8  & 5 & 8 & 10 & $0.37$  \\
& 12  & 6 & 7 & 9 & $0.34$  \\
\midrule
\multirow{3}{*}{Alg.~\ref{alg:det def}, $m=3$} &6  & 7 & 9 & 14 & $0.41$ \\
& 12  & 6 & 8 & 9 & $0.37$  \\
& 18  & 6 & 7 & 8 & $0.48$  \\
\midrule
\multirow{2}{*}{\cite[Alg.~2.3]{Tisseur2010}}&  no piv.& 9 & 21 & 25 & $69.35$  \\
& compl. piv.& 11 & 18 & 21& $65.08$  \\ \hline
 \cite[Alg.~5.1]{Ding2018} & & 8 & 30 & 151 & $1.53$ \\ \midrule
 \cite[Alg.~4.2 and~4.3]{PandurKuzmano}& & 2 & 3 & 4 & $13.72$\\
\bottomrule
\end{tabular}
\end{center}
\end{table}

 For \textbf{type 2} we run Algorithm~\ref{alg:det def} with $m=2$ and $m=3$ and only with exact preconditioners in every iteration: the results  are given in Table~\ref{tab:ex: gen_hyper2 type small n=500}.  We observe that increasing the subspace dimension $p$ may decrease the overall number of iterations, but does not necessarily reduce the execution time of Algorithm~\ref{alg:det def}. We describe in detail the results for $m=3$ and $p=6$. For the first 12 matrix pairs, the algorithm correctly detects the definiteness. 
 For $j=13$ and $j=14$ the algorithm finds some projected pair that is not definite and ends with the conclusion that the original matrix pair is not definite, although it is definite with the length of  definiteness interval of $10^{-13}$ and $10^{-14}$, respectively. The Crawford number for $j=13$ and $j=14$ is approximately $4.45\cdot 10^{-12}$ and $1.86\cdot 10^{-8}$, respectively.  
For $j=15,\ldots,19$ the algorithm correctly detects that the original pair is indefinite by finding some indefinite projected pair. For $j=20$, the algorithm correctly detects in the seventh iteration that the original matrix pair is close to an indefinite one, since the length of the definiteness interval of the last found projected pair is less than $\mathsf{tol}=10^{-12}$. 

 We also run the arc algorithm~\cite[Algorithm~2.3]{Tisseur2010} for \textbf{type 2} for which the   results are given in the middle of Table~\ref{tab:ex: gen_hyper2 type small n=500}. The arc algorithm in this example is very time-consuming. We describe in detail the results with complete pivoting. The arc algorithm detects that the pair is definite for $j=1,\ldots,9$.   For $j\in\{10,\ldots,19\}\setminus\{14,15\}$  detects that the pair is close to indefinite one. For $j=14$, it incorrectly declares that the pair is indefinite but correctly declares that the pair is indefinite for $j=15$ and $j=20$.

 Finally,  the results for running~\cite[Algorithm~5.1]{Ding2018} and \cite[Algorithm~4.2, Algorithm~4.3]{PandurKuzmano} with LOBPCG are at the bottom of Table~\ref{tab:ex: gen_hyper2 type 1 n=500}. In all cases, these algorithms stop when the approximation of the Crawford number is less than $10^{-4}$.
\end{example}


\begin{example}\label{exm:ispitivanje sedlasta}(\textbf{Linear systems in saddle point form})
 The matrix of a linear system in saddle point form has the block structure
\begin{equation}\label{eq: saddle}
    \mathsf{A}=\begin{bmatrix}
     A &B^T\\
     B&-C
\end{bmatrix}
\end{equation}
where $A\in\R^{m_1\times m_1}$ is symmetric positive definite, $B\in\R^{m_2\times m_1}$ with $m_2\leq m_1$, and $C\in\R^{m_2\times m_2}$ is symmetric positive semidefinite. The matrix $\mathsf{A}$ is usually large and sparse, and is indefinite. Its indefiniteness tends to slow down  iterative solvers for linear systems. Nevertheless, if a real scalar $\nu$ is known such that $\mathsf{A}-\nu \mathsf{J}$ with $\mathsf{J}=\diag (I_{m_1},-I_{m_2})$, is positive definite or negative definite, then one can construct a conjugate gradient iteration for solving the linear system $\mathsf{J}\mathsf{A}x=\mathsf{J}b$~\cite{Liesen2008}.

\begin{table}[ht]
\footnotesize
\caption{The results for  detecting   (in)definiteness of the pairs $(\mathsf{A}(\alpha),\mathsf{J})$ of order $834$ for several values of $\alpha$ from Example~\ref{exm:ispitivanje sedlasta}.  The number of iterationss and CPU time are given.}
\label{tab:ispitivanje sedlasta indef.parovi}
\begin{center}
\begin{tabular}{@{}l|rr|rr|rr@{}}
\toprule
Algorithm & $\#$ iter & CPU time & $\#$ iter & CPU time & $\#$ iter & CPU time\\
\midrule
 &\multicolumn{2}{c|}{$\alpha=0.1$} & \multicolumn{2}{c|}{$\alpha=0.3$} & \multicolumn{2}{c}{$\alpha=0.7$}\\ \midrule
 Alg.~\ref{alg:det def}, $m=3$  &&&&&&\\
$p=6$  & 1 & $0.018$ & 5 & $0.039$ & 16 & $0.100$\\
 $p=12$ & 0 & $0.004$ & 4 & $0.041$ & 6 & $0.055$\\
 $p=6$, $B$-orth. & 1 & $0.013$ & 5 & $0.028$ & 7 & $0.043$\\
 $p=12$ $B$-orth. &   0 & $0.003$ & 3 & $0.023$ & 13 & $0.134$\\
\midrule
\cite[Alg.~2.3]{Tisseur2010} & 2 & $0.006$ & 7 & $0.010$ & 14 & $0.054$\\
\midrule
 \cite[Alg.~5.1]{Ding2018} & 76 & $0.508$ & 4 (fail) & $0.017$ & 77 & $0.364$\\
 \midrule
 \cite[Alg.~4.2]{PandurKuzmano} & 3 & $0.097$ & 3 & $0.061$ & 4 & $0.110$\\
\midrule\midrule
  &\multicolumn{2}{c|}{$\alpha=0.72544$} & \multicolumn{2}{c|}{$\alpha=0.72545$} & \multicolumn{2}{c}{$\alpha=1$}\\
   \midrule
   Alg.~\ref{alg:det def}, $m=3$  &&&&&&\\
  $p=6$ & 17 & $0.105$ & 14 & $0.086$ & 6 & $0.037$\\
  $p=12$ &   11 & $0.086$ & 13 & $0.085$ & 6 & $0.045$\\
 $p=6$, $B$-orth. & 18 & $0.099$ & 16 & $0.081$ & 5 & $0.032$\\
 $p=12$ $B$-orth. &   11 & $0.095$ & 12 & $0.088$ & 5 & $0.037$\\
\midrule
\cite[Alg.~2.3]{Tisseur2010}& 18 & $0.110$ & 19 & $0.079$ & 5 & $0.019$\\
\midrule
 \cite[Alg.~5.1]{Ding2018} & 95 (fail) & $0.506$ & 501 (fail) & $2.996$ & 25 (fail) & $0.118$\\
 \midrule
\cite[Alg.~4.2]{PandurKuzmano} & 5 & $0.218$ & 5 & $0.239$ & 4 & $0.122$\\
\bottomrule
\end{tabular}
\end{center}
\end{table}

 In this experiment we consider a saddle point system with a matrix $\mathsf{A}$ of the form~\eqref{eq: saddle} generated by the MATLAB
 package Incompressible Flow Iterative Solution Software (IFFIS)~\cite{ers07,ers14}, version 2.2.\footnote{The latest versions are available from \url{https://personalpages.manchester.ac.uk/staff/david.silvester/ifiss/} } We run the MATLAB script file $\mathtt{stokes\_testproblem}$ to construct sparse matrices $A$, $B$ and $C$, as in~\cite[Experiment 6]{Tisseur2010}. We use the default options, that set up a stabilized discretization  of  a Stokes equation model problem. The matrix  $A$ is of order $m_1=578$ and  $C$ is of order $m_2=256$. 
To have more than one test, we consider the scaled matrix
$$
\mathsf{A}(\alpha)=\begin{bmatrix}
\alpha^2 A & \alpha B^T\\
\alpha B & -C
\end{bmatrix}
$$
for some real  $\alpha>0$.

Table~\ref{tab:ispitivanje sedlasta indef.parovi} reports the number of iterations and the CPU time for the performance of Algorithm~\ref{alg:det def} with $m=3$ and the exact preconditioner in each iteration, the arc algorithm using the MATLAB function $\mathsf{chol}$ for sparse matrices, \cite[Algorithm~5.1]{Ding2018}, and \cite[Algorithm~4.2]{PandurKuzmano} using $\mathsf{eigs}$, applied to $(\mathsf{A}(\alpha),\mathsf{J})$ for several values of $\alpha$.  For    $\alpha\leq 0.72544$ we have indefinite pairs, and  for $\alpha\geq 0.72545$ we have definite pairs.

 The results for our algorithm with two types of orthonormalization are very similar for a fixed $\alpha$ and the same $p$ (except for $\alpha=0.7$.) Our algorithm and the arc algorithm~\cite[Algorithm~2.3]{Tisseur2010} correctly detect the (in)definiteness of the given pairs.  In all cases except for $\alpha=1$,  the enlargement of the subspaces ($p=12$ instead of $p=6$),  gives a smaller number of the attempted Cholesky factorizations in Algorithm~\ref{alg:det def}  with the standard orthonormalization. Notice that the first projected pair for $\alpha=0.1$, $m=3$ and $p=12$ is indefinite and therefore no attempted  Cholesky 
factorizations  is performed. Our algorithm and the arc algorithm have similar number of iterations and CPU time for a fixed value of $\alpha$. 

 For indefinite matrix pairs with $\alpha=0.1, 0.7$, \cite[Algorithm~5.1]{Ding2018} stops when the approximation of the Crawford number is less than $10^{-4}$. For all other cases, it stops when  the eigenvalue approximations lose monotonicity, although the norm of the current residual is not small (it is noted by ``(fail)'' in Table~\ref{tab:ispitivanje sedlasta indef.parovi}), meaning that the Crawford number is not computed to the desired precision.

The results for~\cite[Algorithm~4.2 and 4.3]{PandurKuzmano} using  $\mathsf{eigs}$ are very similar, therefore are given only for~\cite[Algorithm~4.2]{PandurKuzmano}. The indefiniteness for $\alpha=0.1,0.3,0.7$ is correctly detected by~\cite[Algorithm~4.2]{PandurKuzmano} by finding some indefinite projected pair. For $\alpha=0.72544,0.72545$, the algorithm ends when the upper bound for the Crawford number is less than $10^{-4}$. For $\alpha=1$, the algorithm detects the definiteness  by computing the smallest eigenvalue which is positive. 



  \end{example}

\begin{example}\label{exp: banded} (\textbf{Large banded})
    We now test with larger problem size $n=20000$. Matrices $A,B$ are banded with the bandwidth $bw=50$. We form indefinite matrices $A,B$  with the following MATLAB code:
\begin{verbatim}
    A=gallery('lehmer',n);
    d=-bw:bw;
    Aout=spdiags(A,d);
    A=spdiags(Aout,d,n,n);
    nBl=n-(2*bw+1);
    B=spdiags(sprand(nBl,nBl,1),d,nBl,nBl);
    B=blkdiag(diag(d),B); 
    B=(B+B')/2;
\end{verbatim}
 We consider one definite pair of the form $(A-(\lambda_{\min}(A)-1)I_n,B)$, and one indefinite pair $(A,B)$.  We apply Algorithm~\ref{alg:det def} with $m=3$ and $p=6$ ($k_\pm=1$), and three versions of preconditioning: without preconditioning, with exact preconditioning, and with GMRES preconditioning. 
 The results are given in Table~\ref{tab:large banded}. The total number of the iteration is the same when using the standard orthonormalization and when using $B$-orthonormalization. For the latter, the time of execution is given in the brackets.  The (in)definiteness is detected very quickly with all three versions of preconditioning. 
 
 We repeat the experiments with $m=3$, $p=12,18$ and the standard orthonormalization. For the definite pair, the number of the iteration is similar as for $p=6$, but with larger CPU time. For the indefinite pair, the number of the iteration is one in all cases,  with similar CPU time as for $p=6$.  
        \begin{table}[ht]
        \footnotesize
        \caption{Computation results for Example~\ref{exp: banded}.}
        \label{tab:large banded}
        \begin{center}
        \begin{tabular}{@{}l|rr|rr@{}}
        \toprule
        \multirow{2}{*}{Algorithm} &  \multicolumn{2}{ |c| }{definite pair}&  \multicolumn{2}{ |c }{indefinite pair}\\
                & $\#$ iter & CPU time & $\#$ iter  & CPU time  \\
             \midrule
             Alg.~\ref{alg:det def} with $m=3$ and $p=6$ && && \\
            no precond. & 3 & $0.380$ ($0.781$) &3  & $0.377$ ($0.367$)\\
             exact precond. & 5 & $0.795$ ($1.065$)& 1& $0.185$ ($0.208$)\\
               GMRES precond. & 4 & $1.803$ ($1.844$)& 3& $1.646$ ($1.492$)\\
            \midrule
          \cite[Alg.~2.3]{Tisseur2010} & 3& $44.508$ & $7$ & $0.911$\\ \midrule
              \cite[Alg.~5.1]{Ding2018} & 251 (fail) & $40.496$ & 5 (fail) & $0.705$\\
         \midrule
               \cite[Alg.~4.2]{PandurKuzmano} & 1 & $4.837$ & 3 & $4.299$\\
            \cite[Alg.~4.3]{PandurKuzmano}  & 1 &  $4.271$ & 3  & $5.466$\\ \bottomrule
        \end{tabular}  
        \end{center}
    \end{table}


 The arc algorithm and the algorithms from~\cite{PandurKuzmano}   correctly detect the (in)definite\-ness.  
Algorithm~5.1 in~\cite{Ding2018} exhibits certain numerical problems. It stops when the eigenvalue approximations lose monotonicity, although the norm of the current residual is not smaller than the specified tolerance $10^{-8}$ (it is noted by ``(fail)'' in Table~\ref{tab:large banded}).  In the definite case, the Crawford number is approximately $3.332$, and the last computed approximation by~\cite[Algorithm~5.1]{Ding2018} has 6 correct digits.  In the indefinite case the Crawford number is 0, while the last computed approximation is $0.012$. 

\end{example}
\section{Conclusion}\label{sec:concl}

This paper  presents  subspace algorithms, basic one and specialized Algorithm~\ref{alg:det def}, for detecting a definite Hermitian matrix pair. Algorithm~\ref{alg:det def} is an interior eigensolver for computing a few eigenvalues around the definiteness interval if the input is a definite matrix pair. Even if the eigenvalues have not been computed to a reasonable accuracy, it terminates earlier when the definiteness is confirmed by successfully completed Cholesky factorization.  It guarantees to report numerical indefiniteness within a finite number of iterations if the input is not definite. Algorithm~\ref{alg:det def} is backward  stable and has multiple criteria for detecting indefiniteness. 

Numerical experiments show that   Algorithm~\ref{alg:det def}  can be more efficient with the parameter $m=3$ than with $m=2$. For $m=2$, the algorithm benefits significantly from the use of preconditioners,  when a given definite pair is close to an indefinite pair. We can use the preconditioned residual, formed by solving a small number of systems of  linear equations, not necessarily in every iteration, and we can solve the associated linear systems only approximately. 

Algorithm~\ref{alg:det def} is intended for medium-sized or large sparse Hermitian matrix pairs, particularly those with banded matrices, since the Cholesky factorization of a banded matrix and the solution of banded linear systems are computationally inexpensive. The algorithm is able to detect (in)definiteness efficiently, often faster than several existing approaches.
Fast detection of defi\-ni\-teness is important in situations where determining a parameter $\lambda$ such that the matrix $A-\lambda B$ is definite is used as a preprocessing step, for example in the development of algorithms for quadratically constrained quadratic programming ~\cite{Adachi2019,Jiang2019,Jiang2018,Taati2019}, or when a family of Hermitian  matrix pairs $(A(\mu),B(\mu))$ are given and (in)definiteness needs to be detected for many values of $\mu$.

\section*{Acknowledgements}
The author gratefully acknowledges Kre\v{s}imir Veseli\'c and Ding Lu for many valuable comments, and Ding Lu and Fran\c coise Tisseur for providing the MATLAB codes. The author also thanks the anonymous reviewers for their valuable remarks and suggestions, which helped to improve the quality of the paper.

This work was supported by the European union-NextGenerationEU, grant no. 581-UNIOS55, OpHoMat.

 \bibliographystyle{elsarticle-harv} 
 \bibliography{mmp.bib}






\end{document}